\documentclass[11pt]{article}

\usepackage{authblk}

\usepackage{fourier}
\DeclareMathAlphabet{\altmathcal}{OMS}{cmsy}{m}{n}

\usepackage{thmtools} 
\usepackage{thm-restate}

\usepackage[utf8]{inputenc}
\usepackage{amsmath}
\usepackage{mathtools}
\usepackage{amsthm}
\usepackage{amsfonts}
\usepackage{amssymb}
\usepackage[margin=1in]{geometry}
\usepackage[all]{xy}
\usepackage[english]{babel}
\usepackage{graphicx}
\usepackage{arydshln}
\usepackage{tikz}
\usepackage{tikz-cd}
\usepackage{pgfplots}
\usepackage{hyperref}
\usepackage{fancyhdr}
\usepackage{dsfont}
\usepackage{mathrsfs}
\usepackage{float}
\usepackage{multicol}
\usepackage{fix-cm}
\usepackage[shortlabels]{enumitem}
\usetikzlibrary{matrix,arrows,decorations.pathmorphing,shapes.misc}
\usepackage{capt-of}
\usepackage{appendix}
\usepackage{tocloft}
\usepackage[normalem]{ulem}

\cftsetindents{section}{2em}{2.4em}
\cftsetindents{subsection}{4.4em}{3.2em}

\newtheorem{remark}{Remark}[section]
\newtheorem{remarks}{Remarks}[section]
\newtheorem{lemma}{Lemma}[section]
\newtheorem{theorem}{Theorem}[section]

\newtheorem{proposition}{Proposition}[section]

\newtheorem{definition}{Definition}[section]
\newtheorem{fact}{Fact}[section]

\newtheorem{example}{Example}[section]
\newtheorem{corollary}{Corollary}[section]
\newtheorem{caution}{Caution}[section]

\newcommand{\pv}{\mathbb{V}}
\newcommand{\pw}{\mathbb{W}}
\newcommand{\pu}{\mathbb{U}}
\newcommand{\dgm}{\mathrm{Dgm}}

\newcommand{\im}{\textup{im} }
\newcommand{\coim}{\textup{coim} }

\newcommand{\R}{\mathbb{R}}
\newcommand{\N}{\mathbb{N}}

\newcommand{\nmod}{n\textup{-Mod}}
\newcommand{\taumod}[1]{\textup{Mod}_{#1}}
\newcommand{\taumodcat}[1]{\textbf{$\taumod{\text{\boldmath$#1$}}$}}

\newcommand{\type}{\textup{type}}
\newcommand{\lms}{\{\!\!\{}
\newcommand{\rms}{\}\!\!\}}

\newcommand{\mathdszero}{\scalebox{.7}[1]{$\mathbb{O}$}}
\newcommand{\cupdot}{\charfusion[\mathbin]{\cup}{\cdot}}
\newcommand{\bigcupdot}{\charfusion[\mathop]{\bigcup}{\cdot}}
\newcommand{\colim}{\textup{colim}}

\makeatletter
\def\moverlay{\mathpalette\mov@rlay}
\def\mov@rlay#1#2{\leavevmode\vtop{%
		\baselineskip\z@skip \lineskiplimit-\maxdimen
		\ialign{\hfil$\m@th#1##$\hfil\cr#2\crcr}}}
\newcommand{\charfusion}[3][\mathord]{
	#1{\ifx#1\mathop\vphantom{#2}\fi
		\mathpalette\mov@rlay{#2\cr#3}
	}
	\ifx#1\mathop\expandafter\displaylimits\fi}
\makeatother

\newcommand{\Expect}{{\rm I\kern-.3em E}}

\definecolor{darkblue}{rgb}{0.0, 0.0, 0.8}
\definecolor{darkred}{rgb}{0.8, 0.0, 0.0}
\definecolor{darkgreen}{rgb}{0.0, 0.8, 0.0}

\newcommand{\facundo}[1]                {{ \textcolor{darkred} {#1}}}

\begin{document}
	
	\title{Stability of zigzag persistence with respect to a reflection type distance \facundo{(use functor)}}	
	
	\title{The reflection distance between zigzag persistence modules}	
	\author[1]{Alexander Elchesen}
	\author[2]{Facundo M\'emoli}
	\affil[1]{Department of Mathematics,
		The University of Florida\\
		\texttt{a.elchesen@ufl.edu}}
	\affil[2]{Department of Mathematics and Department of Computer Science and Engineering,
		The Ohio State University\\ 
		\texttt{memoli@math.osu.edu}}
	
	\date{\today}
	\maketitle

	\begin{abstract}
		
		By invoking the reflection functors introduced by Bernstein, Gelfand, and Ponomarev in 1973, in this paper we define a metric on the space of all zigzag modules of a given length, which we call the reflection distance. We show that the reflection distance between two given zigzag modules of the same length is an upper bound for the $\ell^1$-bottleneck distance between their respective persistence diagrams. 
		
	\end{abstract}

	\newpage
	\tableofcontents
	\newpage
	
	\section{Introduction and Main Results}
	Persistent Homology is a circle of ideas \cite{frosini1990distance,frosini1992measuring,delfinado1995incremental,robins1999towards,edelsbrunner2000topological,zomorodian2005computing} related to studying the homology of diagrams of simplicial complexes or topological spaces. Often these diagrams are parametrized by a scale parameter which has some geometric meaning. One fundamental example is that given by an increasing sequence of subspaces of a given topological space $X$: $\emptyset=X_0\subset X_1\subset \cdots\subset X_n=X$. In this case, upon passing to homology (with coefficients in a field), one obtains a similar diagram of vector spaces and linear maps $V_0\rightarrow V_1\rightarrow \cdots \rightarrow V_n$. These diagrams are referred to as \emph{persistence modules} and, under mild tameness assumptions, their structure up to isomorphism can be summarized by a multiset of pairs $(i,j)$ with $i\leq j$. The intuition is that these multisets subsume the lifetime of homological features as these are born and are eventually annihilated.

	The poset underlying the diagram of topological spaces $\emptyset=X_0\subset X_1\subset \cdots \subset X_n=X$ from above is simply the poset on $n$ points (generated by) $\bullet \rightarrow \bullet \rightarrow \cdots \rightarrow \bullet$. This setting was generalized by Carlsson and de Silva \cite{zigzag} to allow for any diagram of topological spaces (or simplicial complexes) whose underlying poset is of the form $\bullet \leftrightarrow \bullet \leftrightarrow \cdots \leftrightarrow \bullet$, where at each occurrence of $\leftrightarrow$ exactly one choice for the direction of the arrow is made; the finite sequence of all such choices is called the \emph{type} of $\pv$.
	This generalized setting is called \emph{zigzag persistence}. It provides a complete algebraic invariant for sequences $\pv = (V_i,p_i)$ of vector spaces and linear maps of the form
	\begin{equation*}
	\begin{tikzcd}[row sep=2em,column sep=3.5em]
	V_1\arrow[r,"p_1"]& V_2 \arrow[l]\arrow[r,"p_2"]&\cdots\arrow[l]\arrow[r,"p_{n-2}"]& V_{n-1}\arrow[l]\arrow[r,"p_{n-1}"]& V_n\arrow[l].
	\end{tikzcd}
	\end{equation*} As it was noted in \cite{zigzag} these complete algebraic invariants that zigzag persistence associates to $\pv$ also take the form of persistence diagrams, but these can now be enriched with the type of the zigzag persistence module from which they arise.

	Zigzag persistence has found applications in neuroscience \cite{babichev2017robust,chowdhury2017importance}, and in the analysis of dynamic data \cite{corcoran2016spatio,kim2017stable,chowdhury2017importance}. See \cite{tausz2011applications} for a general description of many possible applications of zigzag persistence.

	In practical applications one typically wishes to use the persistence diagram of a zigzag module to gain insights about the underlying data from which the zigzag module was extracted. With applications in mind, it is important to be able to guarantee stability of zigzag persistence. Informally, a process which takes data as input and provides some invariant as output is \emph{stable} if whenever the input data is perturbed slightly, the resulting invariant changes only slightly. Since data is usually acquired with some inherent noise, stability is a very desirable property.
	
	Depending on the application, a metric $\rho$ is defined on the space of input data and stability results take the form $d_b\leq C \rho$ for some constant $C>0$. A standard metric for measuring the closeness of two persistence diagrams is the bottleneck distance $d_b$ \cite{edelsbrunner2010computational}. In the context of standard persistence, it has been proven that persistence diagrams are stable in different degrees of generality \cite{classical_stability,cohen2006vines,chazal_cohen-steiner_glisse_guibas_oudot_2009,bauer2014induced,lesnick2015theory,cohen2010lipschitz,bubenik2014categorification,bubenik2015metrics}.
	
	An alternative approach is to define a metric \textit{at the algebraic level}, measuring the distance between persistence modules directly. Stability of this form is referred to as \textit{algebraic stability} and it is a notion of algebraic stability that we study in this paper. The algebraic stability of standard persistent homology was studied in \cite{chazal_cohen-steiner_glisse_guibas_oudot_2009} (see also \cite{chazal2016structure,lesnick2015theory,bauer2014induced}), whereas the algebraic stability of zigzag persistence was approached by Botnan and Lesnick through a method different from ours in \cite{algstabzz}.

	In essence, the distance between zigzag modules constructed by Botnan and Lesnick first suitably extrapolates two given zigzag modules into persistence modules over $\mathbb{R}^2$ and then computes an interleaving type distance between these extrapolated modules. They were able to prove that the bottleneck distance between the persistence diagrams of the original persistence modules is bounded above by a constant times the value of the distance between them. A recent refinement by Bjerkevik \cite{bjerkevik} has found the optimal constant for this inequality.
	
	We now describe the structure of our distance and state our main stability result.

	\subsection{Statement of the Main Result}
	
	We introduce here a family of pseudometrics $d_{\altmathcal{R}}^p$, parametrized by $p\in[1,\infty)$, on the space of zigzag modules of length $n$ and then show that in the special case $p = 1$, the inequality $d_b^1\leq d_{\altmathcal{R}}^1$ holds, where $d_b^p$ denotes the \textit{$\ell^p$-bottleneck distance}. The $\ell^p$-bottleneck distance arises from considering a definition analogous to the standard bottleneck distance \cite{edelsbrunner2010computational} with the provision that the ground metric between points is chosen to be the $\ell^p$ norm in $\R^2$ (see details in Section \ref{sec:d-bottleneck}).
	
	For a given $p\in[1,\infty)$, $d_{\altmathcal{R}}^p$ is called the \emph{$p$-reflection distance.} The idea behind the definition of the reflection distance is the following: we consider some collection of transformations of zigzag modules which we will model as a collection $\altmathcal{S}$ of endofunctors on the category $\nmod$ of zigzag modules of a fixed length $n$. For each $p\in[1,\infty)$, we associate a cost to each functor $\altmathcal{F}\in\altmathcal{S}$ by means of a cost function 
	\[ C_p:\altmathcal{S}\to \R^+.\]
	We then define a function $d^p_\altmathcal{R}:\nmod\times\nmod\to \R$ by setting 
	\[ d^p_\altmathcal{R}(\pv,\pw):= \min_{(\altmathcal{F}_1,\altmathcal{F}_2)}\big\{\max\{C_p(\altmathcal{F}_1),C_p(\altmathcal{F}_2)\} \ | \ \altmathcal{F}_1(\pv)\precsim \pw \text{ and } \altmathcal{F}_2(\pw)\precsim \pv\big\},\]
	where the minimum is taken over pairs $(\altmathcal{F}_1,\altmathcal{F}_2)\in \altmathcal{S}\times \altmathcal{S}$ of functors satisfying the conditions that \mbox{$\altmathcal{F}_1(\pv)\precsim \pw$} and $\altmathcal{F}_2(\pw)\precsim \pv$. Here, $\pv_1\precsim\pv_2$ if and only if $\pv_1$ is equivalent to a summand of $\pv_2$, where ``equivalent" refers to equivalence of zigzag modules which differ only in the direction of linear maps representing isomorphisms. If the cost function $C_p$ satisfies the subadditivity condition
	\[ C_p(\altmathcal{F}_2\circ \altmathcal{F}_2)\leq C_p(\altmathcal{F}_1) + C_p(\altmathcal{F}_2)\]
	for all functors $\altmathcal{F}_1, \altmathcal{F}_2\in \altmathcal{S}$ then $d^p_\altmathcal{R}$ turns out to be a pseudometric on $n\textup{-Mod}$.
	
	Of course, $d^p_\altmathcal{R}$ depends both on the collection of functors which we restrict ourselves to and the cost function $C_p$ used. The functors which we will restrict ourselves to in this paper are defined by replacing certain subdiagrams of a given zigzag module by a diagram formed from its limit or colimit. Such functors are closely related to the reflection functors of Bernstein, Gelfand, and Ponomarev \cite{bgp}, hence the name \textit{the reflection distance}. The cost function chosen simply counts the number of transformations needed to transform a pair of zigzag modules into each other, weighted by the parameter $p$. Our main result is then the following

	\begin{restatable}[Main Theorem]{theorem}{thmmain}
		\label{thm:main}For all zigzag modules $\pv,\pw \in n\textup{-Mod}$ we have
		\[d_{b}^1(\dgm(\pv),\dgm(\pw)) \leq d_{\altmathcal{R}}^1(\pv,\pw).\]
	\end{restatable}
	\noindent Here, $d_b^1$ denotes the \textit{$\ell^1$-bottleneck distance} as described above.

	\subsection{Organization of the paper}
	Section \ref{sec:zigzag-mods} recalls the main facts about zigzag modules that we will need  in this paper; Section  \ref{sec:decompositions} recalls elements regarding the decomposition of zigzag modules as direct sums of interval modules; Section \ref{sec:type-transformations} sets terminology which will be used in later sections to describe transformations between different types of zigzag modules; Section \ref{sec:ref-functors} describes reflection functors and their effect on interval modules; Section \ref{sec:ref-distance} provides the precise description of our reflection distance. In Section \ref{sec:d-bottleneck} we describe the $\ell^p$-bottleneck distance and give the proof of our main theorem, Theorem \ref{thm:main}; Section \ref{sec:discussion} provides an overview of some lines of related research which may be of interest. Finally, Appendices \ref{app:cats} and \ref{app:matchings} contain background material on Category Theory and  Matchings.

	\section{Zigzag Modules} \label{sec:zigzag-mods}
	
	Fix a field $\mathbb{F}$. All vector spaces throughout will be finite dimensional over $\mathbb{F}$. A \textit{zigzag module} $\pv = (V_i,p_i)$ is a finite sequence 
	
	\begin{equation}\label{zzmod}
	\begin{tikzcd}[row sep=2em,column sep=3.5em]
	V_1\arrow[r,"p_1"]& V_2 \arrow[l]\arrow[r,"p_2"]&\cdots\arrow[l]\arrow[r,"p_{n-2}"]& V_{n-1}\arrow[l]\arrow[r,"p_{n-1}"]& V_n\arrow[l]
	\end{tikzcd}
	\end{equation}
	of vector spaces and linear transformations between them. An arrow $V_i \stackrel{p_i}{\longleftrightarrow} V_{i+1}$ represents either a forward linear map $V_i \stackrel{p_i}{\longrightarrow} V_{i+1}$ or a backward linear map $V_{i} \stackrel{p_i}{\longleftarrow} V_{i+1}$, but never both. The $p_i$ are referred to as \textit{structure maps}. Note that the $i^{\textup{th}}$ structure map $p_i$ either has domain $V_i$ or $V_{i+1}$ and codomain $V_{i+1}$ or $V_{i}$ respectively. We will use the notation $p_i:V_{i_1}\to V_{i_2}$ when the direction of $p_i$ has not been specified. In other words, $i_1,i_2\in \{i,i+1\}$ with $i_1\neq i_2$, $\textup{dom}(p_i) = V_{i_1}$, and $\textup{cod}(p_i) = V_{i_2}$.
	
	The \textit{length} of a zigzag module is the length of the sequence (\ref{zzmod}) above. We will denote the collection of all zigzag modules of length $n$ by $\nmod$. A finite sequence $\tau$ of the symbols $\rightarrow$ and $\leftarrow$, indicating the directions of the linear maps in (\ref{zzmod}) as read from left to right, is called the \textit{type} of the zigzag module. Formally, the type of a zigzag module of length $n$ is a sequence $\tau\in\{\rightarrow,\leftarrow\}^{n-1}$. We will use the notation $\altmathcal{T}_n := \{\rightarrow,\leftarrow\}^{n-1}$.
	\begin{remark} Sequences in $\altmathcal{T}_n$ have length $n-1$, not length $n$. This is so that zigzag modules in $\nmod$ have types in $\altmathcal{T}_n$.
	\end{remark}
	
	We define a map 
	\[\type = \type_n:\nmod\to \altmathcal{T}_n\]
	where $\textup{type}(\pv)$ is the type of $\pv$. We denote the collection of all zigzag modules of type $\tau$ by $\taumod{\tau}$, that is\footnote{Note that we are slightly abusing notation since $\taumod{\tau}$ may not be a set.}
	\[\taumod{\tau}:= \{\pv\in \nmod \ | \ \type(\pv) = \tau\}.\]
	Note that $\taumod{\tau} = \type^{-1}(\tau)$ and $\nmod = \bigcup_{\tau\in\altmathcal{T}_n}\taumod{\tau}$. A zigzag module of type $\tau$ is also called a \textit{$\tau$-module}.
	
	\begin{example} Let $n = 3$. Then
		\[ \altmathcal{T}_n = \{\underbrace{(\rightarrow,\rightarrow)}_{\tau_1},\underbrace{(\rightarrow,\leftarrow)}_{\tau_2},\underbrace{(\leftarrow,\rightarrow)}_{\tau_3},\underbrace{(\leftarrow,\leftarrow)}_{\tau_4}\}.\]
		Consider the zigzag modules
		\begin{equation*}
		\begin{aligned}[c]
		\pv_1 &= 	\begin{tikzcd}[column sep=2em]
		\mathbb{F}\arrow[r,"\textup{id}"]& \mathbb{F} \arrow[r,"\textup{id}"]& \mathbb{F}  
		\end{tikzcd}\\
		\pv_3 &= 
		\begin{tikzcd}[column sep=2em]
		\mathbb{F}&\arrow[l,swap,"0"]0\arrow[r,"0"] & \mathbb{F}  
		\end{tikzcd}\\
		\end{aligned}\quad
		\begin{aligned}[c]
		\pv_2 &= \begin{tikzcd}[column sep=2em]
		\mathbb{F}\arrow[r,"\textup{id}"]& \mathbb{F} & 0  \arrow[l,swap,"0"]
		\end{tikzcd}\\
		\pv_4 &= \begin{tikzcd}[column sep=2em]
		0\arrow[r,leftarrow,"0"]& \mathbb{F}& \mathbb{F}  \arrow[l,swap,"\textup{id}"]
		\end{tikzcd}\\
		\end{aligned}
		\end{equation*}
		where $\mathbb{F}$ is viewed as a $1$-dimensional vector space over itself and $0$ denotes the trivial vector space. Each of these zigzag modules is an element of $\nmod$ and we have $\type(\pv_j) = \tau_j$ for $j = 1, 2,3 ,4$.
	\end{example}
	
	\subsection{Morphisms Between Zigzag Modules}
	
	Fix zigzag modules $\pv = (V_i,p_i)$ and $\pw = (W_i,q_i)$ in $\taumod{\tau}$. A \textit{morphism} from $\pv$ to $\pw$ is a collection $\phi = \{\phi_i:V_i\to W_i\}_{i = 1}^n$ of linear transformations such that the diagram
	\[
	\begin{tikzcd}[row sep=3em,column sep=4em]
	V_1\arrow[r,"p_1"]\arrow[d,"\phi_1"]& V_2 \arrow[l]\arrow[r,"p_2"]\arrow[d,"\phi_2"]&\cdots\arrow[l]\arrow[r,"p_{n-2}"]& V_{n-1}\arrow[l]\arrow[r,"p_{n-1}"]\arrow[d,"\phi_{n-1}"]& V_n\arrow[l]\arrow[d,"\phi_{n}"]\\
	W_1\arrow[r,"q_1"]& W_2 \arrow[l]\arrow[r,"q_2"]&\cdots\arrow[l]\arrow[r,"q_{n-2}"]& W_{n-1}\arrow[l]\arrow[r,"q_{n-1}"]& V_n\arrow[l]
	\end{tikzcd}
	\]
	commutes. We denote a morphism from $\pv$ to $\pw$ by $\phi:\pv\to\pw$. The linear maps $\phi_i$ comprising the morphism $\phi$ are called the \textit{components} of $\phi$. Composition of morphisms is defined component-wise and the identity morphism $\textup{id}_\pv:\pv\to\pv$ is the morphism all of whose components are identity maps. With these definitions in place, for every $n\in\mathbb{N}$ and for each $\tau\in\altmathcal{T}_n$, the collection $\taumod{\tau}$ of $\tau$-modules together with the collection of all morphisms between them forms a category denoted $\taumodcat{\tau}$. We call a morphism $\phi$ an \textit{isomorphism}, \textit{monomorphism}, or \textit{epimorphism} if all of the $\phi_i$ are either bijective, injective, or surjective, respectively. If there exists an isomorphism between $\pv$ and $\pw$ we say that $\pv$ and $\pw$ are \textit{isomorphic} and write $\pv\cong \pw$.

	Given a $\tau$-module $\pv = (V_i,p_i)$, a \textit{submodule} $\pw = (W_i,q_i)$ of $\pv$ is a $\tau$-module such that for all $i$ either $p_i(W_i)\subseteq W_{i+1}$ or $p_i(W_{i+1})\subseteq W_i$, depending on whether $p_i:W_i\to W_{i+1}$ or $p_i:W_{i+1}\to W_{i}$, respectively. In this case, the linear map $q_{i}$ is taken to be the restrictions of the map $p_{i}$ to the subspace $W_i$ or $W_{i+1}$, again depending on the direction of $p_i$. A $\tau$-module $\pw$ is isomorphic to a submodule of $\pv$ if and only if there exists a monomorphism $f:\pw\to\pv$. This justifies working with monomorphisms between zigzag modules instead of working with submodules directly. We write $\pw\leq \pv$ whenever there exists a monomorphism $f:\pw\to \pv$. Note that $\pw\cong \pv$ if and only if $\pw\leq \pv$ and $\pv\leq \pw$.
	
	\subsection{Interval Modules and the Zero Module}
	Fix $n\in\mathbb{N}$ and $\tau\in\altmathcal{T}_n$. For each pair $b,d\in\{1,\hdots,n\}$ with $b\leq d$ we define a zigzag module $\mathbb{I}_{\tau}([b,d]) = (I_i,p_i)\in\taumod{\tau}$, called the \textit{interval $\tau$-module} on $[b,d]$, by setting 
	\[ I_i:= \begin{cases} \mathbb{F} & b\leq i\leq d\\ 0 & \textup{otherwise}\end{cases}\quad \text{ and }\quad p_i := \begin{cases} \textup{id}_{\mathbb{F}} & b\leq i <d\\ 0 & \textup{otherwise}.\end{cases}\]
	When the type is fixed we will drop the subscript $\tau$ and just write $\mathbb{I}([b,d])$. Interval $\tau$-modules of the form $\mathbb{I}_\tau([k,k])$ are called \textit{simple interval $\tau$-modules}. We also define the zero module $\mathdszero_\tau = (Z_i,z_i)$ of type $\tau$ to be the $\tau$-module with $Z_i = 0$ and $z_i = 0$ for all $i$.
	
	\begin{example} Let $n = 3$ and let $\tau = (\rightarrow,\leftarrow)\in\altmathcal{T}_n$. There are $6 = \binom{n+1}{2}$ nonzero interval $\tau$-modules given by 
		
		\begin{equation*}
		\begin{aligned}[c]
		\mathbb{I}_\tau([1,1]) &= 	\begin{tikzcd}[column sep=2em]
		\mathbb{F}\arrow[r,"0"]& 0 & 0  \arrow[l,swap,"0"]
		\end{tikzcd}\\
		\mathbb{I}_\tau([1,3]) &= 	\begin{tikzcd}[column sep=2em]
		\mathbb{F}\arrow[r,"\textup{id}"]& \mathbb{F} & \mathbb{F}  \arrow[l,swap,"\textup{id}"]
		\end{tikzcd}\\
		\mathbb{I}_\tau([2,3]) &= 
		\begin{tikzcd}[column sep=2em]
		0\arrow[r,"0"]& \mathbb{F} & \mathbb{F}  \arrow[l,swap,"\textup{id}"]
		\end{tikzcd}\\
		\end{aligned}\quad
		\begin{aligned}[c]
		\mathbb{I}_\tau([1,2]) &= \begin{tikzcd}[column sep=2em]
		\mathbb{F}\arrow[r,"\textup{id}"]& \mathbb{F} & 0  \arrow[l,swap,"0"]
		\end{tikzcd}\\
		\mathbb{I}_\tau([2,2]) &= \begin{tikzcd}[column sep=2em]
		0\arrow[r,"0"]& \mathbb{F} & 0  \arrow[l,swap,"0"]
		\end{tikzcd}\\
		\mathbb{I}_\tau([3,3]) &= \begin{tikzcd}[column sep=2em]
		0\arrow[r,"0"]& 0 & \mathbb{F}  \arrow[l,swap,"0"]
		\end{tikzcd}\\
		\end{aligned}
		\end{equation*}
	\end{example}

	\section{Decompositions of Zigzag Modules} \label{sec:decompositions}
	In this section we define the direct sum of zigzag modules of the same type and state a standard unique decomposition theorem, an adaptation of the Krull-Remak-Schmidt theorem to the context of zigzag persistence. Combined with Gabriel's Theorem which characterizes the indecomposable zigzag modules as precisely the interval zigzag modules, we are able to define \textit{persistence diagrams}, an object of fundamental importance in persistence theory.
	\subsection{Indecomposables and Summands}
	The \textit{direct sum} of two $\tau$-modules $\mathbb{X} = (X_i,\alpha_i)$ and $\mathbb{Y} = (Y_i,\beta_i)$ is a $\tau$-module $\mathbb{X}\oplus\mathbb{Y} = (Z_i,\gamma_i)$ where $Z_i = X_i\oplus Y_i$ and where $\gamma_i = \alpha_i\oplus \beta_i$ for all $i$. We say that $\pw$ is a \textit{summand} of $\pv$ whenever there exists a $\tau$-module $\pu$ such that $\pv \cong \pw\oplus\pu$ and we write $\pw\preceq \pv$. The relation $\preceq$ defines a partial order on the isomorphism classes of $\taumod{\tau}$. For later reference, we record the following easy proposition describing the structure maps of a summand:
	\begin{proposition} \label{types_of_summand_maps}If $(W_i,q_i) = \pw\preceq\pv = (V_i,p_i)$ and if the $k^\textup{th}$ structure map $p_k$ of $\,\pv$ is  injective, surjective, or bijective then the $k^\textup{th}$ structure map $q_k$ of $\,\pw$ is also injective, surjective, or bijective, respectively.
		%
		%
	\end{proposition}
	
	A $\tau$-module $\pv$ is said to be \textit{decomposable} if there exists nonzero $\tau$-modules $\pw$ and $\pu$ such that $\pv \cong \pw\oplus \pu$ and is said to be \textit{indecomposable} otherwise. The following important theorem says that every zigzag module decomposes as a sum of indecomposable modules and characterizes the indecomposables as the interval modules:
	\begin{theorem}[Krull-Remak-Schmidt, Gabriel \cite{zigzag},\cite{gabe}]\label{krull}
		For each $n\in\N$ and for every $\tau\in\altmathcal{T}_n$, the indecomposable $\tau$-modules are precisely the interval $\tau$-modules. Moreover, every $\pv\in\taumod{\tau}$ decomposes as a direct sum of interval $\tau$-modules. This decomposition is unique up to the order in which the summands appear.
	\end{theorem}
	\subsection{Persistence  Diagrams}
	Fix an $n\in \N$ and a type $\tau\in \altmathcal{T}_n$. By Theorem \ref{krull}, every $\tau$-module $\pv\in\taumod{\tau}$ has a decomposition of the form 
	\begin{equation}\label{decomp}
	\pv \cong \mathbb{I}_\tau([b_1,d_1])\oplus\cdots\oplus\mathbb{I}_\tau([b_N,d_N]),\end{equation}
	this decomposition being unique up to the ordering of the summands.	We define the \textit{persistence diagram} of $\pv$ to be the multiset
	\[\dgm(\pv) : = \lms(b_i,d_i)\in\N\times \N \ | \ 1\leq i\leq N\rms,\]
	whose elements are ordered pairs of endpoints defining the interval modules in the decomposition (\ref{decomp}).
	In particular, we always have the decomposition 
	\begin{equation} \label{persdecomp}
	\pv \cong \bigoplus_{(b,d)\in \dgm(\pv)}\mathbb{I}_\tau([b,d]).\end{equation}
	Persistence diagrams thus characterize zigzag modules. That is, for fixed type $\tau$, a $\tau$-module determines and is determined up to isomorphism by its persistence diagram.
	
	There is a simple relationship between the persistence diagram of a zigzag module and the persistence diagram of any of its summands:
	\begin{proposition}\label{perssum}
		Fix $n\in \N$ and $\tau\in \altmathcal{T}_n$. If $\,\pw,\pv\in\taumod{\tau}$ with $\pw\preceq \pv$ then $\dgm(\pw)\subseteq\dgm(\pv)$.
	\end{proposition}
	\begin{proof}
		Since $\pw\preceq \pv$ there exists $\pu\in\taumod{\tau}$ such that $\pv\cong \pw\oplus \pu$. Using the decomposition (\ref{persdecomp}), we have
		\[ \pv \cong \left( \bigoplus_{(b,d)\in\dgm(\pw)}\mathbb{I}([b,d])\right)\oplus \left( \bigoplus_{(b,d)\in\dgm(\pu)}\mathbb{I}([b,d])\right).\]
		By the uniqueness statement of Theorem \ref{krull}, $\dgm(\pv) = \dgm(\pw)\sqcup\dgm(\pu)$ so that $\dgm(\pw)\subseteq\dgm(\pv)$.
	\end{proof}

	\section{Type Transformations and Arrow Reversals} \label{sec:type-transformations}
	In this section we define several transformations of types, i.e. maps $\altmathcal{T}_n\to\altmathcal{T}_n$, whose purposes are mainly formal; they will serve to define the appropriate domains and codomains for the transformations of zigzag modules defined in Section \ref{sect_ref_func}.
	\subsection{Sinks and Sources} Fix $n\in \N$ and let $\tau\in \altmathcal{T}_n$.
	A zigzag module $\mathds{V} = (V_i,p_i)\in\taumod{\tau}$ has a \textit{sink} at index $k\in\{2,\hdots,n-1\}$ if it has the form 
	\[	\mathds{V} = \begin{tikzcd}[row sep=2em,column sep=3.5em]
	\cdots \arrow[r,"p_{k-2}"]& V_{k-1} \arrow[l]\arrow[r,"p_{k-1}"]& V_{k}& V_{k+1}\arrow[r,"p_{k+1}"]\arrow[l,swap,"p_k"]& \cdots\arrow[l].
	\end{tikzcd}\]
	In addition, we say that $\mathds{V}$ has a \textit{sink} at index 1 or index $n$ if the maps $p_1$ or $p_{n-1}$ are of the form $V_{1} \stackrel{p_1}{\longleftarrow} V_{2}$ or $V_{n-1} \stackrel{p_{n-1}}{\longrightarrow} V_{n}$, respectively.
	
	Similarly, $\mathds{V} = (V_i,p_i)\in\taumod{\tau}$ has a \textit{source} at index $k\in\{2,\hdots,n-1\}$ if it has the form 
	\[	\mathds{V} = \begin{tikzcd}[row sep=2em,column sep=3.5em]
	\cdots \arrow[r,"p_{k-2}"]& V_{k-1} \arrow[l]&\arrow[l,swap,"p_{k-1}"] V_{k}\arrow[r,"p_{k}"]& V_{k+1}\arrow[r,"p_{k+1}"]& \cdots\arrow[l],
	\end{tikzcd}\]
	and has a \textit{source} at index $1$ or $n$ if the maps $p_1$ or $p_n$ are of the form $V_{1} \stackrel{p_1}{\longrightarrow} V_{2}$ or $V_{n-1} \stackrel{p_{n-1}}{\longleftarrow} V_{n}$, respectively. Equivalently, a $\tau$-module $\mathds{V}$ has a sink at index $k\in\{1,\hdots,n\}$ if none of the linear maps $p_i$ have domain $V_k$, and $\mathds{V}$ has a source at index $k$ if none of the linear maps $p_i$ have codomain $V_k$.
	
	Note that the property of having a sink or source at a given index depends only on the type of the zigzag module in question; that is, if a $\tau$-module $\mathds{V}$ has a sink or source at index $k\in\{1,\hdots,n\}$ then any other $\tau$-module will also have, respectively, a sink or source at index $k$. This leads to the following definitions: 
	\begin{definition}\label{sink_source_flow}We say a type $\tau\in\altmathcal{T}_n$ has a \textit{sink} or \textit{source} at index $k$ if any (and hence every) $\tau$-module has, respectively, a sink or source at index $k$. Otherwise, $\tau$ is said to have a \textit{flow} at index $k$. If both the $(k-1)^\textup{st}$ and $k^\textup{th}$ entries of $\tau$ are $\rightarrow$, then $\tau$ is said to have a $\textit{forward flow}$, and a flow which is not a forward flow is called a \textit{backwards flow}.
	\end{definition}
	
	\begin{caution}
		When we make mention of a type $\tau\in\altmathcal{T}_n$ as having a sink or source at index $k\in\{1,\hdots,n\}$, we mean that any $\tau$-module has a sink or source, respectively, at index $k$. The word ``index" here \textbf{does not} refer to the $k^{\textup{th}}$ component of the sequence of arrows defining $\tau$.
	\end{caution}
	\subsection{Type Transformations}
	\begin{definition} \label{type_reversal}Fix $n\in \N$. For each $k\in\{1,\hdots,n-1\}$, we define the \textit{$k^\textup{th}$ reversal map}
		\[ r_k:\altmathcal{T}_n\to \altmathcal{T}_n\]
		which maps a type $\tau$ to a type $r_k\tau$ whose $k^{th}$ entry is obtained by reversing the $k^\textup{th}$ entry of $\tau$.
	\end{definition}
	\begin{example}
		Let $n = 4$ and $\tau = (\rightarrow,\leftarrow,\rightarrow)\in\altmathcal{T}_n$. We have
		\[r_1\tau = (\leftarrow,\leftarrow,\rightarrow), \quad r_2\tau = (\rightarrow,\rightarrow,\rightarrow), \quad r_3\tau = (\rightarrow,\leftarrow,\leftarrow).\]
	\end{example} 
	
	\begin{definition}\label{type_extroversion}
		For fixed $n$ and for each $k\in\{1,\hdots,n\}$, we define the \textit{$k^\textup{th}$ extroversion map}
		\[ \sigma_k:\altmathcal{T}_n\to\altmathcal{T}_n\]
		which maps a type $\tau$ to a type $\sigma_k\tau$ which is obtained by placing a \textit{source} at index $k$ in $\tau$.
	\end{definition}
	More precisely, for $k\in\{2,\hdots,n-1\}$, $\sigma_k\tau$ is obtained from $\tau$ by replacing the $(k-1)^\textup{st}$ and $k^\textup{th}$ entries of $\tau$ by $\leftarrow$ and $\rightarrow$, respectively. The type $\sigma_1\tau$ is obtained by replacing the $1^\textup{st}$ entry of $\tau$ by $\rightarrow$ and the type $\sigma_n\tau$ is obtained by replacing the $(n-1)^\textup{st}$ entry of $\tau$ by $\leftarrow$. Similarly, we make the following:
	
	\begin{definition} \label{type_introversion} For fixed $n$ and for each $k\in\{1,\hdots,n\}$, we define the \textit{$k^\textup{th}$ introversion map} 
		\[ \zeta_k:\altmathcal{T}_n\to\altmathcal{T}_n\]
		which maps a type $\tau$ to the type $\zeta_k\tau$ which is obtained by placing a \textit{sink} at index $k$ in $\tau$.
	\end{definition}
	
	\begin{example}\label{type_intro_extro_example}
		Let $n = 4$ and $\tau = (\rightarrow,\rightarrow,\leftarrow)\in\altmathcal{T}_n$. Then we have
		\begin{equation*}
		\begin{aligned}[c]
		\sigma_1\tau = \tau & = (\rightarrow,\rightarrow,\leftarrow)  \\
		\sigma_3\tau & = (\rightarrow,\leftarrow,\rightarrow)
		\end{aligned}\quad
		\begin{aligned}[c]
		\sigma_2\tau & = (\leftarrow,\rightarrow,\leftarrow)\\
		\sigma_4\tau = \tau& = (\rightarrow,\rightarrow,\leftarrow).
		\end{aligned}
		\end{equation*}
		and
		\begin{equation*}
		\begin{aligned}[c]
		\zeta_1\tau & = (\leftarrow,\rightarrow,\leftarrow)  \\
		\zeta_3\tau = \tau & = (\rightarrow,\rightarrow,\leftarrow)
		\end{aligned}\quad
		\begin{aligned}[c]
		\zeta_2\tau & = (\rightarrow,\leftarrow,\leftarrow)\\
		\zeta_4\tau & = (\rightarrow,\rightarrow,\rightarrow).
		\end{aligned}
		\end{equation*}
	\end{example}
	\subsection{Arrow Reversals}
	We wish to identify zigzag modules $\pv,\pw\in\nmod$ which differ only in the direction of arrows representing isomorphisms. For example, the zigzag modules
	\[	\begin{tikzcd}[column sep=2em]
	0\arrow[r,"0"]&\mathbb{F}\arrow[r,"\textup{id}"]& \mathbb{F}\arrow[r,"0"] & 0 
	\end{tikzcd}\quad \textup{ and } \quad	\begin{tikzcd}[column sep=2em]
	0\arrow[r,"0"]&\mathbb{F}\arrow[r,leftarrow,"\textup{id}"]& \mathbb{F}\arrow[r,"0"] & 0 \end{tikzcd}\]
	contain the same information and we wish to regard them as equivalent. The goal of this section is to establish notation for dealing with zigzag modules which are to be regarded as equivalent in this way.
	
	\begin{definition}
		Fix $\tau\in\altmathcal{T}_n$. For each $k\in\{1,\hdots,n-1\}$ we define $\taumod{\tau}^{\textup{iso},k}\subset\taumod{\tau}$ by setting 
		\[\taumod{\tau}^{\textup{iso},k} : = \{\pv =  (V_i,p_i) \ | \ p_k \textup{ is an isomorphism}\}.\]
		Recall the type reversal map $r_k:\altmathcal{T}_n\to\altmathcal{T}_n$ of Definition \ref{type_reversal} which reverses the $k^\textup{th}$ arrow of a given $\tau\in \altmathcal{T}_n$. We define a map 
		\[\altmathcal{A}_k:\taumod{\tau}^{\textup{iso},k}\to \taumod{r_k\tau}^{\textup{iso},k}\]
		by setting $\altmathcal{A}_k(\pv) = (V_i,q_i)\in\taumod{r_k\tau}$, where $q_i = p_i$ for $i\neq k$ and $q_k = p_k^{-1}$.
	\end{definition}
	
	\begin{remarks}\label{arr_rev_remarks} \hfill
		\begin{enumerate}  \item If $V_{i} \stackrel{p}{\longrightarrow} V_{j}$ is an isomorphism appearing in $\pv$, and $V_j\stackrel{q}{\longleftrightarrow} V_k$ is an adjacent arrow, then we have the following isomorphisms of zigzag modules:
			\[
			\begin{tikzcd}[row sep=3em,column sep=4em]
			\cdots \arrow[r]& V_i\arrow[l]\arrow[r,"p"]\arrow[d,"\textup{id}"]&V_j\arrow[r,"q"]\arrow[d,"p^{-1}"]& V_k \arrow[l]\arrow[r]\arrow[d,"\textup{id}"]&\cdots\arrow[l]\\
			\cdots \arrow[r]& V_i\arrow[l]\arrow[r,"\textup{id}"]&V_i\arrow[r,"\alpha"]& V_k \arrow[l]\arrow[r]&\cdots\arrow[l]
			\end{tikzcd}
			\]
			where $\alpha = qp$ if $V_j\stackrel{q}{\longrightarrow}V_k$ and $\alpha = p^{-1}q$ if $V_j\stackrel{q}{\longleftarrow}V_k$ (and all vertical maps not appearing are identities). Thus it is safe to assume without loss of generality that all structure maps which are isomorphisms are in fact identities, in which case $\altmathcal{A}_k(\pv)$ is obtained by changing the type of $\,\pv$ from $\tau$ to $r_k\tau$ but leaving all of the vector spaces $V_i$ and linear maps $p_i$ unchanged.
			
			\item Note that $\altmathcal{A}_k\circ\altmathcal{A}_k = \textup{id}_{\taumod{\tau}^{\textup{iso},k}}$ since reversing the direction of an isomorphism twice leaves the zigzag module unchanged.
			
		\end{enumerate}
	\end{remarks}
	
	We make the following observations about arrow reversals:
	\begin{proposition}\label{arrow_reversal_properties}
		Let $\tau\in\altmathcal{T}_n$ and suppose that $\pv\in\taumod{\tau}^{\textup{iso},i}$ for some $i\in\{1,\dots,n-1\}$.
		\begin{enumerate}[(1)]
			\item If $\,\pv\cong \pw$ then $\pw\in\taumod{\tau}^{\textup{iso},i}$ and $\altmathcal{A}_i(\pv)\cong \altmathcal{A}_i(\pw)$,
			\item If $\,\pw\in\taumod{\tau}^{\textup{iso},i}$ then $\pv\oplus\pw\in\taumod{\tau}^{\textup{iso},i}$ and $\altmathcal{A}_i(\pv\oplus\pw) = \altmathcal{A}_i(\pv)\oplus\altmathcal{A}_i(\pw)$,
			\item If $\,\pw\preceq \pv$ then $\pw\in\taumod{\tau}^{\textup{iso},i}$ and $\altmathcal{A}_i(\pw)\preceq\altmathcal{A}_i(\pv)$,
			\item If $\,\pv\in\taumod{\tau}^{\textup{iso},i}\cap\taumod{\tau}^{\textup{iso},j}$ for some $j\in\{1,\dots,n-1\}$ then $\altmathcal{A}_i\altmathcal{A}_j(\pv) = \altmathcal{A}_j\altmathcal{A}_i(\pv)$.
		\end{enumerate}
		\begin{proof}
			(1) That $\pw\in\taumod{\tau}^{\textup{iso},i}$ is immediate from the commutativity relations imposed on the components of a $\tau$-module isomorphism. Moreover, it is easily verified that the components of any isomorphism $\phi:\pv\to\pw$ of $\tau$-modules will also serve as the components of an isomorphism of $r_i\tau$-modules between $\altmathcal{A}_i(\pv)$ and $\altmathcal{A}_i(\pw)$.
			
			(2) This follows immediately after noticing that $(p_i\oplus q_i)^{-1} = p_i^{-1}\oplus q_i^{-1}$, where $p_i$ and $q_i$ denote the $i^\textup{th}$ structure maps of $\pv$ and $\pw$, respectively.
			
			(3) If $\pw\preceq\pv\in\taumod{\tau}^{\textup{iso},i}$ then the $i^\textup{th}$ structure map of $\pv$ is an isomorphism so that, by Proposition \ref{types_of_summand_maps}, the $i^\textup{th}$ structure map of $\pw$ is an isomorphism as well and hence $\pw\in\taumod{\tau}^{\textup{iso},i}$. Now if $\pu\in \taumod{\tau}$ is such that $\pv \cong \pw\oplus\pu$ then by parts (1) and (2), we have
			\[\altmathcal{A}_i(\pv)\cong \altmathcal{A}_i(\pw\oplus\pu) = \altmathcal{A}_i(\pw)\oplus\altmathcal{A}_i(\pu),\]
			and thus $\altmathcal{A}_i(\pw)\preceq\altmathcal{A}_i(\pv)$.
			
			(4) If $i = j$ then the result follows immediately. If $i\neq j$, then the result follows by noting that $\altmathcal{A}_i$ and $\altmathcal{A}_j$ operate on different linear maps $p_i$ and $p_j$, so that the order in which they are applied does not matter.
		\end{proof}
	\end{proposition}

	\subsection{Equivalence of Zigzag Modules}
	We now define an equivalence relation on $\nmod$, formalizing our discussion at the beginning of the previous section. For $\pv,\pw\in\nmod$ we write $\pv\sim\pw$ if and only if either $\pv\cong \pw$ or there is a finite sequence $k_1,\hdots,k_j$ of indices in $\{1,\hdots,n-1\}$ such that 
	\[\pw\cong \altmathcal{A}_{k_j}\altmathcal{A}_{k_{j-1}}\cdots\altmathcal{A}_{k_1}(\pv).\]
	In words, $\pv\sim\pw$ if $\pw$ can be obtained, up to isomorphism, from $\pv$ by reversing some (possibly empty) set of arrows representing isomorphisms. Reflexivity of the relation $\sim$ is clear, while symmetry  follows from Remark \ref{arr_rev_remarks} (2) and transitivity from Proposition \ref{arrow_reversal_properties} (1). Thus $\sim$ does indeed define an equivalence relation on the isomorphism classes of $n\textup{-Mod}$.

	\begin{remark}
		Note that $\mathdszero_\tau\sim\mathdszero_{\tau'}$ for any types $\tau,\tau'\in\altmathcal{T}_n$. Hence, in what follows we drop the subscript indicating type and denote the zero module of any type by $\mathdszero$.
	\end{remark}
	
	The proof of the next proposition is sketched out by Oudot in \cite{oudotbook}; we give the full details here:
	\begin{proposition}[\cite{oudotbook}]\label{simpers}
		Let $\pv,\pw\in\nmod$. If $\,\pv\sim\pw$ then \[\dgm(\pv) = \dgm(\pw).\]
		\begin{proof} Let $\tau = \type(\pv)$ and write $\pv \cong \bigoplus_{(b,d)\in\dgm(\pv)}\mathbb{I}_\tau([b,d])$. Since $\pv\sim\pw$, there is a sequence of indices $k_1,k_2,\hdots,k_j\in\{1,\hdots,n-1\}$ such that
			\[\pw\cong \altmathcal{A}_{k_j}\altmathcal{A}_{k_{j-1}}\cdots\altmathcal{A}_{k_1}(\pv).\]
			Let $\tau' = r_{k_j}r_{k_{j-1}}\cdots r_{k_1}\tau$ so that $\type(\pw) = \tau'$ and consider the zigzag module $\pu\in\taumod{\tau'}$ defined by
			\[\pu:= \bigoplus_{(b,d)\in\dgm(\pv)}\mathbb{I}_{\tau'}([b,d]).\]
			By definition of $\pu$, we have $\dgm(\pu) = \dgm(\pv)$. We claim that $\pu\cong\pw$. To see this, notice that if $V_{i} \stackrel{p}{\longrightarrow} V_{j}$ is a structure map of $\pv$ with $p$ being an isomorphism and $i,j$ being consecutive integers in $\{1,\hdots,n\}$, and if $d :=\textup{dim}(V_i) = \textup{dim}(V_{j})$, then there is an isomorphism $\psi:V_i\to \bigoplus_{m=1}^d\mathbb{F}$ so that the diagram 
			\[\begin{tikzcd}[row sep=2.5em,column sep=3.5em]
			V_i\arrow[r,"p"]\arrow[d,"\psi"]& V_j \arrow[d,"\psi\circ p^{-1}"]\\
			\bigoplus_{m = 1}^{d}\mathbb{F}\arrow[r,"\textup{id}"]& \bigoplus_{m = 1}^{d}\mathbb{F}
			\end{tikzcd}\]
			commutes. The above diagram commutes if and only if the diagram
			\[\begin{tikzcd}[row sep=2.5em,column sep=3.5em]
			V_i\arrow[d,"\psi"]& V_j \arrow[l,swap,"p^{-1}"]\arrow[d,"\psi\circ p^{-1}"]\\
			\bigoplus_{m = 1}^{d}\mathbb{F}&\arrow[l,swap,"\textup{id}"] \bigoplus_{m = 1}^{d}\mathbb{F}
			\end{tikzcd}\]
			commutes. Applying this principle to every square at which an arrow reversal is applied, we see that $\pu\cong \pw$. Hence $\dgm(\pv) = \dgm(\pu) = \dgm(\pw)$ by Theorem \ref{krull}.
		\end{proof}
	\end{proposition}

	\begin{definition}
		We define a relation $\precsim$ on $\nmod$ by declaring $\pw\precsim \pv$ if and only if there exists a zigzag module $\pw'\in\nmod$ with $\pw\sim \pw'$ and $\pw'\preceq\pv$.
	\end{definition} 
	In words, $\pw\precsim\pv$ if we can obtain a summand of $\pv$ by reversing any number of the arrows of $\pw$ representing isomorphisms.
	
	\begin{proposition}\label{precsim_preorder} $\precsim$ is a preorder on $\nmod$. Moreover, $\pv\sim \pw$ if and only if $\pw\precsim\pv$ and $\pv\precsim\pw$ so that $\precsim$ induces a partial order on $\nmod/\sim$.
		\begin{proof}
			Since $\pv\sim\pv$ and $\pv\preceq\pv$, we have $\pv\precsim\pv$. If $\pv_1\precsim\pv_2$ and $\pv_2\precsim\pv_3$ then there are zigzag modules $\pw_1$ and $\pw_2$ such that $\pv_1\sim\pw_1\preceq\pv_2$ and $\pv_2\sim\pw_2\preceq\pv_3$. That is, there are compositions of arrow reversals $\altmathcal{A}$ and $\altmathcal{B}$ such that $\altmathcal{A}(\pv_1) \cong\pw_1\preceq \pv_2$ and $\altmathcal{B}(\pv_2)\cong\pw_2\preceq \pv_3$. Then by Proposition \ref{arrow_reversal_properties} parts (1) and (3), we have $\altmathcal{B}\altmathcal{A}(\pv_1) \preceq\altmathcal{B}(\pv_2)\cong\pw_2\preceq \pv_3$ so that $\pv_1\precsim\pv_3$. This shows that $\precsim$ is a preorder.
			
			Now if $\pv\sim\pw$ then $\pv\sim\pw\preceq\pw$ and $\pw\sim\pv\preceq\pv$ so that $\pw\precsim\pv$ and $\pv\precsim\pw$. Conversely, if $\pw\precsim\pv$ and $\pv\precsim\pw$ then there exists zigzag modules $\pw'\preceq\pw$ and $\pv'\preceq\pv$ such that $\pw\sim\pv'$ and $\pv\sim\pw'$. Let $\pu_1,\pu_2\in\nmod$ be such that
			\[ \pv\cong \pv'\oplus\pu_1 \quad \textup{ and } \quad \pw\cong \pw'\oplus\pu_2.\]
			Since $\pw\sim\pv'$, there is a composition of arrow reversals $\altmathcal{A}$ such that $\altmathcal{A}(\pw)\cong \pv'$, and similarly there is a composition of arrow reversals $\altmathcal{B}$ such that $\altmathcal{B}(\pv)\cong \pw'$. Hence $\pw\cong \pw'\oplus\pu_2 \cong \altmathcal{B}(\pv)\oplus \pu_2$. Then we have
			\begin{align*}\pv' & \cong \altmathcal{A}(\pw) \cong \altmathcal{A}(\altmathcal{B}(\pv)\oplus\pu_2)\\
			& = \altmathcal{A}(\altmathcal{B}(\pv))\oplus\altmathcal{A}(\pu_2) \cong \altmathcal{A}(\altmathcal{B}(\pv'\oplus\pu_1))\oplus\altmathcal{A}(\pu_2)\\
			& = \altmathcal{A}(\altmathcal{B}(\pv'))\oplus \altmathcal{A}(\altmathcal{B}(\pu_1))\oplus \altmathcal{A}(\pu_2).
			\end{align*}
			Using the above isomorphisms together with Propositions \ref{perssum} and \ref{simpers}, we have
			\begin{align*}
			\dgm(\pv') & = \dgm(\altmathcal{A}(\altmathcal{B}(\pv'))\oplus \altmathcal{A}(\altmathcal{B}(\pu_1))\oplus \altmathcal{A}(\pu_2))\\
			& = \dgm(\altmathcal{A}(\altmathcal{B}(\pv')))\sqcup\dgm(\altmathcal{A}(\altmathcal{B}(\pu_1)))\sqcup\dgm(\altmathcal{A}(\pu_2)))\\
			& = \dgm(\pv')\sqcup\dgm(\pu_1)\sqcup\dgm(\pu_2)
			\end{align*}
			so that $\dgm(\pu_1) = \emptyset = \dgm(\mathbb{U}_2)$. Hence $\mathbb{U}_1\sim\mathdszero\sim\mathbb{U}_2$ so that in fact
			\[ \pv\cong \pv'\quad \textup{ and } \quad \pw\cong \pw'.\]
			Thus $\pv\sim \pw$, completing the proof.			
		\end{proof}
	\end{proposition}

	\section{Reflection Functors}\label{sect_ref_func} \label{sec:ref-functors}
	In this section, we define \textit{reflection functors}, first introduced by Bernstein, Gelfand, and Ponomarev in \cite{bgp} as a tool for proving Gabriel's Theorem. Reflection functors were more recently applied to zigzag persistence by Kalisnik in \cite{kalisnik_thesis} to give an alternative proof of Carlsson and de Silva's Diamond Principle \cite{zigzag}.
	
	In the language of category theory, reflection functors send a diagram in the category of vector spaces to a new diagram in that same category, obtained by replacing particular subdiagrams by \textit{universal cones} or \textit{cocones}. We refer to the reader to Appendix \ref{app:cats} for a review of the necessary categorical notions.
	
	\subsection{Reflections on Zigzags}\label{sec:reflections}
	Suppose that $\pv = (V_i,p_i)\in\taumod{\tau}$ for some $\tau\in \altmathcal{T}_n$. For $k\in \{2,\hdots,n-1\}$, we isolate the subdiagram
	\begin{equation}\label{subdiagram}
	\begin{tikzcd}[row sep=2em,column sep=3.5em]
	V_{k-1}\arrow[r,"p_{k-1}"]& V_k \arrow[l]\arrow[r,"p_k"] &V_{k+1}\arrow[l].
	\end{tikzcd}
	\end{equation}
	We compute the limit $(L_k,\lambda_j)$ of the extracted diagram and then consider the new diagram
	\begin{equation} \label{extro_zigzag}
	\begin{tikzcd}[row sep=2em,column sep=3.5em]
	V_{k-1}&\arrow[l,swap,"\lambda_{k-1}"] L_k \arrow[r,"\lambda_k"] &V_{k+1},
	\end{tikzcd}
	\end{equation}
	and then define $\altmathcal{L}_k(\pv)\in\taumod{\sigma_k\tau}$ to be the zigzag module obtained by replacing the appearance of the subdiagram (\ref{subdiagram}) in $\pv$ by diagram (\ref{extro_zigzag}). That is, a map
	\[\altmathcal{L}_k:\taumod{\tau}\to \taumod{\sigma_k\tau}\]
	is specified by the following diamond diagram relating $\pv$ and $\altmathcal{L}_k(\pv)$:
	\[\begin{tikzcd}[row sep=1.5em,column sep=2em]	\pv = \\
	\\ \altmathcal{L}_k(\pv) := 
	\end{tikzcd}
	\begin{tikzcd}[row sep=2em,column sep=3em]	&&&\arrow[dl] V_k&&&\\	V_1\arrow[r,"p_1"]& 
	\cdots \arrow[r,"p_{k-2}"]\arrow[l]& V_{k-1} \arrow[l]\arrow[ur,"p_{k-1}"]&  & V_{k+1}\arrow[r,"p_{k+1}"]\arrow[ul]\arrow[ul,swap,leftarrow,"p_k"]& \cdots\arrow[l]\arrow[r,"p_{n-1}"]&V_n\arrow[l]\\
	&&& L_k\arrow[ul,"\lambda_{k-1}"]\arrow[ur,swap,"\lambda_k"]&&&
	\end{tikzcd}\quad\quad 
	\]
	Similarly, if $(C_k,\gamma_j)$ is the colimit of diagram (\ref{subdiagram}), then we consider the new diagram
	\begin{equation} \label{intro_zigzag}
	\begin{tikzcd}[row sep=2em,column sep=3.5em]
	V_{k-1}\arrow[r,"\gamma_{k-1}"] &C_k &\arrow[l,swap,"\gamma_k"] V_{k+1}.
	\end{tikzcd}
	\end{equation}
	and define $\altmathcal{C}_k(\pv)\in\taumod{\zeta_k\tau}$ to be the zigzag module obtained by replacing the appearance of the subdiagram (\ref{subdiagram}) in $\pv$ by diagram (\ref{intro_zigzag}). Thus, we have a map
	\[\altmathcal{C}_k:\taumod{\tau}\to \taumod{\zeta_k\tau}\]
	for which $\pv$ and $\altmathcal{C}_k(\pv)$ are related by the diamond diagram
	\[\begin{tikzcd}[row sep=1.5em,column sep=2em]	\pv = \\
	\\ \altmathcal{C}_k(\pv) := 
	\end{tikzcd}
	\begin{tikzcd}[row sep=2em,column sep=3em]	&&&\arrow[dl] V_k&&&\\
	V_1\arrow[r,"p_1"]& \cdots \arrow[r,"p_{k-2}"]\arrow[l]& V_{k-1} \arrow[l]\arrow[ur,"p_{k-1}"]&  & V_{k+1}\arrow[r,"p_{k+1}"]\arrow[ul]\arrow[ul,swap,leftarrow,"p_k"]& \cdots\arrow[l]\arrow[r,"p_{n-1}"]&V_n\arrow[l]\\
	&&& C_k\arrow[ul,leftarrow,"\gamma_{k-1}"]\arrow[ur,swap,leftarrow,"\gamma_k"]&&&
	\end{tikzcd}\quad\quad\]
	In order to define reflections at indices $1$ and $n$, we consider the diagrams
	\[
	\begin{tikzcd}[row sep=2em,column sep=3.5em]
	0\arrow[r,"0"]&V_{1}\arrow[l]\arrow[r,"p_{1}"]& V_2\arrow[l]
	\end{tikzcd}
	\qquad \textup{ and } \qquad
	\begin{tikzcd}[row sep=2em,column sep=3.5em]
	V_{n-1}\arrow[r,"p_{n-1}"]& V_n\arrow[l]\arrow[r,"0"]& 0 \arrow[l]
	\end{tikzcd}
	\]
	and their limits or colimits. These reflections thus depend on the choice of direction for the zero map. We wish to allow the flexibility of choosing the direction of this map each time a reflection is applied at index $1$ or $n$. Thus, we define reflections $\altmathcal{L}_1^\rightarrow$ and $\altmathcal{L}_1^\leftarrow$ obtained by replacing $V_1$ and $p_1$ with the limit of the diagram
	\[	
	\begin{tikzcd}[row sep=2em,column sep=3.5em]
	0\arrow[r,"0"]&V_{1}\arrow[r,"p_{1}"]& V_2\arrow[l]
	\end{tikzcd}
	\qquad \textup{ or } \qquad
	\begin{tikzcd}[row sep=2em,column sep=3.5em]
	0&\arrow[l,swap,"0"]V_{1}\arrow[r,"p_{1}"]& V_2\arrow[l]
	\end{tikzcd}
	\]
	respectively. The maps $\altmathcal{L}_n^\rightarrow$, $\altmathcal{L}_n^\leftarrow$, 
	$\altmathcal{C}_1^\rightarrow$, 
	$\altmathcal{C}_1^\leftarrow$,   $\altmathcal{C}_n^\rightarrow$, and $\altmathcal{C}_n^\leftarrow$ are all defined analogously.
	
	For each $k\in \{2,\dots,n-1\}$, we let $\altmathcal{R}_k$ denote an unspecified choice of $\altmathcal{L}_k$ or $\altmathcal{C}_k$. Similarly $\altmathcal{R}_1$ denotes an unspecified choice of $\altmathcal{L}_1^\rightarrow$, $\altmathcal{L}_1^\leftarrow$, $\altmathcal{C}_1^\rightarrow$, or $\altmathcal{C}_1^\leftarrow$ and $\altmathcal{R}_n$ denotes an unspecified choice of $\altmathcal{L}_n^\rightarrow$, $\altmathcal{L}_n^\leftarrow$, $\altmathcal{C}_n^\rightarrow$, or $\altmathcal{C}_n^\leftarrow$.
	\subsection{Functoriality of Reflections}
	Let $\pv,\pw\in\taumod{\tau}$ and let $\phi:\pv\to\pw$ be a morphism of $\tau$-modules. Denote the limits of the diagrams $V_{k-1}\stackrel{p_{k-1}}{\longleftrightarrow}V_k\stackrel{p_k}{\longleftrightarrow} V_{k+1}$ and $W_{k-1}\stackrel{q_{k-1}}{\longleftrightarrow}W_k\stackrel{q_k}{\longleftrightarrow} W_{k+1}$ by $(L_k^V,\lambda_j^V)$ and $(L_k^W,\lambda_j^W)$, respectively. Then we have the following commutative diagram:
	\[\footnotesize
	\begin{tikzcd}[row sep=2.25em,column sep=4.25em]
	& L_k^V \arrow[ddl,swap,thick,black,"\lambda_1^V"] \arrow[ddr,thick,black,"\lambda_2^V"] \arrow[dd,gray,"\lambda_3^V"]
	\arrow[drr,dashed,thick,black,"\exists!\mu"] \\
	&&& L_W \arrow[ddr,thick,black,"\lambda_3^W"] \\
	V_{k-1} \arrow[drr,thick,black,swap,"\phi_{k-1}"] 
	\arrow[r,gray,"p_{k-1}"] &{\color{gray}V_k}\arrow[l,gray] \arrow[r,gray,"p_k"]\arrow[drr,gray,"\phi_k"]&
	V_{k+1}\arrow[l,gray] \arrow[drr,thick,black,"\phi_{k+1}"] \\
	&& W_{k-1}
	\arrow[r,swap,gray,"q_{k-1}"]
	\arrow[uur,leftarrow,crossing over,thick,black,"\lambda_1^W"]  & {\color{gray}W_k}\arrow[r,swap,gray,"q_{k}"]\arrow[l,gray]&
	W_{k+1}\arrow[l,gray]
	\end{tikzcd}\]
	From this diagram, we see that $(L_k^V,\gamma_j^V)$ is a cone over $W_{k-1}\stackrel{q_{k-1}}{\longleftrightarrow}W_k\stackrel{q_k}{\longleftrightarrow} W_{k+1}$ so that by the universality of $L_k^W$, there exists a unique linear transformation $\mu:L_k^V\to L_k^W$ making the diagram commute. We then define $\psi = \altmathcal{L}_k(\phi):\altmathcal{L}_k(\pv)\to\altmathcal{L}_k(\pw)$ by setting $\psi_j = \phi_j$ for all $j\neq k$ and $\phi_k := \mu$. The fact that $\psi$ is a well-defined morphism follows from the commutativity of the boldened portion of the diagram above. Appealing to duality, we similarly obtain a morphism $\altmathcal{C}_k(\phi):\altmathcal{C}_k(\pv)\to\altmathcal{C}_k(\pw)$.

	\begin{definition} For each $k\in\{1,\hdots,n\}$ and for each $\tau\in\altmathcal{T}_n$, the above definitions make $\altmathcal{L}_k$ into a functor from $\taumodcat{\tau}$ to $\taumodcat{\sigma_k\tau}$, which we call the \textit{extroversion reflection functor} at index $k$. Similarly, $\altmathcal{C}_k$ is a functor from $\taumodcat{\tau}$ to $\taumodcat{\zeta_k\tau}$, which we call the \textit{introversion reflection functor} at index $k$.
	\end{definition}
	
	\begin{remark}
		When a zigzag module $\pv$ has a \textit{sink} at index $k$, i.e., when $\pv$ has the form 
		\[	\pv = \begin{tikzcd}[row sep=2em,column sep=3.5em]
		\cdots \arrow[r,"p_{k-2}"]& V_{k-1} \arrow[l]\arrow[r,"p_{k-1}"]& V_{k}& V_{k+1}\arrow[r,"p_{k+1}"]\arrow[l,swap,"p_k"]& \cdots\arrow[l].
		\end{tikzcd}\]
		then, up to isomorphism, $\altmathcal{L}_k(\pv)$ is the same as the zigzag module obtained by applying the sink reflection functor of \cite{bgp} to $\pv$. Dually, when $\pv$ has as source at index $k$ then the zigzag module obtained by applying $\,\altmathcal{C}_k$ to $\pv$ is isomorphic to that obtained by applying the source reflection functor.
	\end{remark}

	\subsection{Properties of Reflection Functors}
	A number of the results below hold for categorical reasons, and we will rely on several high-level results for their proofs. Statements and proofs of these general categorical results are contained in Appendix \ref{app:cats}. We will denote by $\mathbf{vect}_{\mathbb{F}}$ the category of finite-dimensional vector spaces over the field $\mathbb{F}$.
	\begin{proposition}\label{refsub}
		Fix $k\in\{1,\hdots,n\}$, let $\tau\in \altmathcal{T}_n$, and let $\pv,\pw\in\taumod{\tau}$. If $\,\pw\leq \pv$ then $\altmathcal{R}_k(\pw)\leq \altmathcal{R}_k(\pv)$. Furthermore, if $\,\pv\cong\pw$ then $\altmathcal{R}_k(\pv)\cong\altmathcal{R}_k(\pw)$.
		\begin{proof} If $\pw\leq \pv$ then there exists a monomorphism $j:\pw\hookrightarrow \pv$. By functoriality of $\altmathcal{R}_k$ we obtain a morphism $\altmathcal{R}_k(j):\altmathcal{R}_k(\pw)\rightarrow \altmathcal{R}_k(\pv)$. Theorem \ref{mono_diagrams_mono_lim} of the appendix implies that all of the components of $\altmathcal{R}_k(j)$ are monomorphisms in $\mathbf{vect}_{\mathbb{F}}$, i.e., are injective. Thus $\altmathcal{R}_k(j)$ is a monomorphism from $\altmathcal{R}_k(\pw)$ to $\altmathcal{R}_k(\pv)$ so that $\altmathcal{R}_k(\pw)\leq\altmathcal{R}_k(\pv)$. The second statement follows from the fact that $\pv\cong\pw$ if and only if $\pw\leq\pv$ and $\pv\leq\pw$.
		\end{proof}
	\end{proposition}

	\begin{proposition}\label{refsum} 	Fix $k\in\{1,\hdots,n\}$, let $\tau\in \altmathcal{T}_n$, and let $\pv,\pw\in\taumod{\tau}$. Then $\altmathcal{R}_k(\pv\oplus \pw) \cong \altmathcal{R}_k(\pv)\oplus\altmathcal{R}_k(\pw)$. This statement generalizes to the sum of any finite number of $\tau$-modules.
		\begin{proof}
			Viewing zigzag modules as diagrams in $\mathbf{vect}_{\mathbb{F}}$, the result follows from Theorem \ref{lim_prod_is_prod_lim}. The result extends to arbitrary finite sums by induction.
		\end{proof}
	\end{proposition}
	
	\begin{corollary}\label{refsummand}
		Fix $k\in\{1,\hdots,n\}$, let $\tau\in \altmathcal{T}_n$, and let $\pv$, $\pw\in\taumod{\tau}$. If $\,\pw\preceq \pv$ then $\altmathcal{R}_k(\pw)\preceq\altmathcal{R}_k(\pv)$.
		\begin{proof}If $\pw\preceq\pv$ then $\pv \cong \pw\oplus\pu$ for some $\pu\in\taumod{\tau}$ so that by Propositions \ref{refsub} and \ref{refsum}, $$\altmathcal{R}_{k}(\pv)\cong \altmathcal{R}_k(\pw\oplus\pu) \cong \altmathcal{R}_k(\pw)\oplus\altmathcal{R}_k(\pu)$$ and hence $\altmathcal{R}_k(\pw)\preceq\altmathcal{R}_k(\pv)$.
		\end{proof}
	\end{corollary}
	
	\begin{corollary}\label{refprecsim}
		Fix $k\in\{1,\hdots,n\}$, let $\tau\in \altmathcal{T}_n$, and let $\pv$, $\pw\in\nmod$. If $\,\pw\precsim \pv$ then $\altmathcal{R}_k(\pw)\precsim\altmathcal{R}_k(\pv)$.
		\begin{proof}
			If $\pw\precsim \pv$ then $\pw\sim\pw'\preceq \pv$ for some $\pw'$. Now note that that if $\pw\sim \pw'$ then $\altmathcal{R}_k(\pw)\sim \altmathcal{R}_k(\pw')$. This follows from the fact that limits and colimits of diagrams are unaltered if any number of arrows representing isomorphisms are reversed. From this fact and Corollary \ref{refsummand}, we have \[\altmathcal{R}_k(\pw)\sim\altmathcal{R}_k(\pw')\preceq \altmathcal{R}_k(\pv),\]
			and hence $\altmathcal{R}_k(\pw)\precsim\altmathcal{R}_k(\pv)$.
		\end{proof}
	\end{corollary}
	
	\begin{theorem}\label{bigsumreflection}
		Fix $k\in\{1,\hdots,n\}$, let $\tau\in \altmathcal{T}_n$, and let $\pv\in\taumod{\tau}$. Then we have
		\[\altmathcal{R}_k(\pv) \cong \bigoplus_{(b,d)\in\dgm(\pv)}\altmathcal{R}_k(\mathbb{I}_\tau([b,d])).\]
		\begin{proof}
			Write $\pv \cong \bigoplus_{(b,d)\in\dgm(\pv)}\mathbb{I}_\tau([b,d])$. Using Propositions \ref{refsub} and \ref{refsum}, we have
			\[\altmathcal{R}_k(\pv) \cong \altmathcal{R}_k\left(\bigoplus_{(b,d)\in\dgm(\pv)}\mathbb{I}_\tau([b,d])\right)
			\cong \bigoplus_{(b,d)\in\dgm(\pv)}\altmathcal{R}_k(\mathbb{I}_\tau([b,d])).\qedhere\]
		\end{proof}
	\end{theorem}
	
	Theorem \ref{bigsumreflection} together with the following theorems describe exactly how reflections of zigzag modules effect their persistence diagrams:
	\begin{theorem}[The Diamond Principle Part I, \cite{bgp},\cite{zigzag},\cite{kalisnik_thesis}]\label{intervalreflections}
		Let $k\in\{2,\hdots,n-1\}$ and let $\tau\in\altmathcal{T}_n$ have a sink at index $k$ so that $\sigma_k\tau$ has a source at index $k$. Then the reflection functors $\altmathcal{L}_k$ and $\altmathcal{C}_k$ induce mutually inverse bijections between the isomorphism classes of interval $\tau$-modules and the isomorphism classes of interval $\sigma_k\tau$-modules, with the exception of the simple interval modules $\mathbb{I}_{\tau}([k,k])$ and $\mathbb{I}_{\sigma_k\tau}([k,k])$ which are annihilated by these functors.\footnote{We warn the reader that while we are referring to isomorphism classes of interval $\tau$ modules, in order to keep our notation simple, we will use the notation $\mathbb{I}_\tau([b,d])$ instead  of the more correct notation $[\mathbb{I}_\tau([b,d])]$. } These functors act on the (isomorphism classes of) interval modules as follows:
		
		\begin{equation*}
		\begin{aligned}
		\mathbb{I}_\tau([k,k])\, &&\longrightarrow & \quad \mathdszero &&\\
		\mathdszero && \longleftarrow \,\;& \quad \mathbb{I}_{\sigma_k\tau}([k,k])&&\\
		\mathbb{I}_\tau([b,k-1])\, &&\longleftrightarrow & \quad\mathbb{I}_{\sigma_k\tau}([b,k]) && \quad \text{for }b\leq k-1\\
		\mathbb{I}_\tau([b,k])\, &&\longleftrightarrow & \quad\mathbb{I}_{\sigma_k\tau}([b,k-1]) && \quad \text{for }b\leq k-1\\
		\mathbb{I}_\tau([k+1,d])\, &&\longleftrightarrow & \quad\mathbb{I}_{\sigma_k\tau}([k,d]) && \quad \text{for }d\geq k+1\\
		\mathbb{I}_\tau([k,d])\, &&\longleftrightarrow & \quad\mathbb{I}_{\sigma_k\tau}([k+1,d]) && \quad \text{for }d\geq k+1\\
		\mathbb{I}_\tau([b,d])\, &&\longleftrightarrow & \quad\mathbb{I}_{\sigma_k\tau}([b,d]) && \quad \text{otherwise}.\\	
		\end{aligned}
		\end{equation*}
		For $k = 1$, we have that $\altmathcal{L}_1^\rightarrow$ and $\altmathcal{C}_1^\leftarrow$ are mutually inverse except on the simple interval modules $\mathbb{I}_\tau([1,1])$ and $\mathbb{I}_{\sigma_1\tau}([1,1])$:
		\begin{equation*}
		\begin{aligned}
		\mathbb{I}_\tau([1,1])\, &&{\longrightarrow} & \quad \mathdszero &&\\
		\mathdszero && \longleftarrow \,\; & \quad \mathbb{I}_{\sigma_1\tau}([1,1])&&\\
		\mathbb{I}_\tau([1,d])\, &&\longleftrightarrow & \quad\mathbb{I}_{\sigma_1\tau}([2,d]) && \quad \text{for }d\geq 2\\
		\mathbb{I}_\tau([2,d])\, &&\longleftrightarrow & \quad\mathbb{I}_{\sigma_1\tau}([1,d]) && \quad \text{for }d\geq 2\\
		\end{aligned}
		\end{equation*}
		and for $k = n$, we have that $\altmathcal{L}_n^\leftarrow$ and $\altmathcal{C}_n^\rightarrow$ are mutually inverse except on the simple interval modules $\mathbb{I}_\tau([n,n])$ and $\mathbb{I}_{\sigma_1\tau}([n,n])$:
		\begin{equation*}
		\begin{aligned}
		\mathbb{I}_\tau([n,n]) &&{\longrightarrow} & \quad \mathdszero &&\\
		\mathdszero && \longleftarrow \,\; & \quad  \mathbb{I}_{\sigma_n\tau}([n,n])&&\\
		\mathbb{I}_\tau([b,n])\, &&\longleftrightarrow & \quad\mathbb{I}_{\sigma_n\tau}([b,n-1]) && \quad \text{for }b\leq n-1.\\
		\mathbb{I}_\tau([b,n-1])\, &&\longleftrightarrow & \quad\mathbb{I}_{\sigma_n\tau}([b,n]) && \quad \text{for }b\leq n-1.\\
		\end{aligned}
		\end{equation*}
		In the above, arrows pointing from left to right denote the effect of applying $\altmathcal{L}_k$ (or $\altmathcal{L}_k^{\leftarrow}, \altmathcal{L}_k^{\rightarrow}$) while arrow pointing from right to left denote the effect of applying $\altmathcal{C}_k$ (or $\altmathcal{C}_k^{\leftarrow}, \altmathcal{C}_k^{\rightarrow}$).		
	\end{theorem}
	A statement completely analogous to Theorem \ref{intervalreflections} holds in the case that $\tau$ has a source at index $k$, with the roles of $\altmathcal{L}_k$ and $\altmathcal{C}_k$ reversed, $\altmathcal{L}_1^\rightarrow$, $\altmathcal{C}_1^\leftarrow$, $\altmathcal{L}_n^\leftarrow$, $\altmathcal{C}_n^\rightarrow$ replaced with $\altmathcal{C}_1^\leftarrow$, $\altmathcal{L}_1^\rightarrow$, $\altmathcal{C}_n^\rightarrow$, $\altmathcal{L}_n^\leftarrow$ respectively, and with $\sigma$ begin replaced with $\zeta$, but with the actions on interval modules being otherwise identical.
	
	\begin{figure}
		\centering
		\scalebox{1.5}{
			\begin{tikzpicture}
			\draw[-latex,very thick] (1,0.5)--(3.5,0.5);\draw[-latex,very thick] (0.5,1)--(0.5,3.5);
			
			\draw[dashed,very thick] (0.5,0.5)--(1,0.5);
			\draw[dashed,very thick] (0.5,0.5)--(0.5,1);
			\draw[very thick] (0.479,0.5)--(0.5,0.5);
			
			\fill[gray]  (1.5,2) circle [radius=1.75pt];
			\fill[gray]  (1,2.5) circle [radius=1.75pt];
			\fill[gray]  (1.5,2) circle [radius=1.75pt];\fill[gray]  (2,2) circle [radius=1.75pt];
			\fill[gray]  (1,1.5) circle [radius=1.75pt];\fill[gray]  (1,2) circle [radius=1.75pt];
			\fill[gray]  (1.5,1.5) circle [radius=1.75pt];\fill[gray]  (1,1) circle [radius=1.75pt];
			\fill[gray]  (2.5,3) circle [radius=1.75pt];\fill[gray]  (2,3) circle [radius=1.75pt];
			\fill[gray]  (1.5,3) circle [radius=1.75pt];\fill[gray]  (1,3) circle [radius=1.75pt];
			\fill[gray] (2.5,2.5) circle [radius=1.75pt];\fill[gray] (2,2.5) circle [radius=1.75pt];
			\fill[gray]  (1.5,2.5) circle [radius=1.75pt];\fill[gray]  (3,3) circle [radius=1.75pt];
			
			\tikzset{cross/.style={cross out, draw=red, fill=none, minimum size=3, inner sep=0pt, outer sep=0pt}, cross/.default={2pt}}
			
			\node[draw,cross out,red,inner sep = 1pt,minimum size = 4.5pt] at (2,2){};
			
			\draw[-latex,semithick] (1.5,1.8)--(1.5,1.5);
			\draw[-latex,semithick] (1.5,1.7)--(1.5,2);
			
			\draw[-latex,semithick] (1.0,1.8)--(1.0,1.5);
			\draw[-latex,semithick] (1.0,1.7)--(1.0,2);
			
			\draw[-latex,semithick] (2.2,2.5)--(2.5,2.5);
			\draw[-latex,semithick] (2.3,2.5)--(2,2.5);
			
			\draw[-latex,semithick] (2.2,3.0)--(2.5,3.0);
			\draw[-latex,semithick] (2.3,3.0)--(2,3.0);
			
			\footnotesize
			\draw[very thin] (2,0.55) -- (2,0.45)node[anchor=north ] {$k$};
			\draw[very thin] (2.5,0.55) -- (2.5,0.45);
			\draw[very thin] (3,0.45) -- (3,0.55);
			
			\draw[very thin] (1,0.55) -- (1,0.45);
			\draw[very thin] (1.5,0.55) -- (1.5,0.45);
			
			\footnotesize
			\draw[very thin] (.55,2) -- (.45,2)node[anchor=east ] {$k$};
			\draw[very thin] (.55,2.5) -- (.45,2.5);
			\draw[very thin] (.55,3) -- (0.45,3);
			\draw[very thin] (.55,1) -- (0.45,1);
			\draw[very thin] (.55,1.5) -- (.45,1.5);
			\end{tikzpicture}\quad
			
			\begin{tikzpicture}
			
			\draw[-latex,very thick] (1,0.5)--(3.5,0.5);\draw[very thick] (0.5,0.5)--(0.5,1.5);
			
			\draw[dashed,very thick] (0.5,0.5)--(1,0.5);
			
			\draw[very thick,black] (1,1.25)--(3,1.25);
			\draw[very thick,orange] (1,1.0)--(2,1.0);
			\draw[very thick,red] (2,.75)--(3,.75);
			
			\footnotesize
			\draw[very thin] (2,0.55) -- (2,0.45)node[anchor=north ] {$k$};
			\draw[very thin] (2.5,0.55) -- (2.5,0.45);
			\draw[very thin] (3,0.45) -- (3,0.55);
			\draw[very thin] (1,0.55) -- (1,0.45);
			\draw[very thin] (1.5,0.55) -- (1.5,0.45);
			
			\draw[-latex,very thick] (1,-1.5)--(3.5,-1.5);\draw[very thick] (0.5,-1.5)--(0.5,-0.5);
			
			\draw[dashed,very thick] (0.5,-1.5)--(1,-1.5);
			
			\draw[very thick,black] (1,-0.75)--(3,-0.75);
			\draw[very thick,orange] (1,-1)--(1.5,-1);
			\draw[very thick,red] (2.5,-1.25)--(3,-1.25);
			
			\footnotesize
			\draw[very thin] (2,-1.45) -- (2,-1.55)node[anchor=north ] {$k$};
			\draw[very thin] (2.5,-1.45) -- (2.5,-1.55);
			\draw[very thin] (3,-1.55) -- (3,-1.45);
			\draw[very thin] (1,-1.45) -- (1,-1.55);
			\draw[very thin] (1.5,-1.45) -- (1.5,-1.55);
			\normalsize
			
			\node[anchor = west] at (1.73,-.25) {$\Updownarrow$ \scriptsize$\altmathcal{R}_k$};
			
			\end{tikzpicture}\quad}
		\caption{Left: Points in the diagram of a zigzag module $\pv$ move according to the arrows when the reflection functors $\altmathcal{L}_k$ or $\altmathcal{C}_k$ are applied to diagrams with sinks or sources, respectively, at index $k$. The point $(k,k)$ corresponding to the simple summand $\mathbb{I}([k,k])$ is killed. Right: An analogous picture for barcodes.}
		\label{reflection_diagram}
	\end{figure}
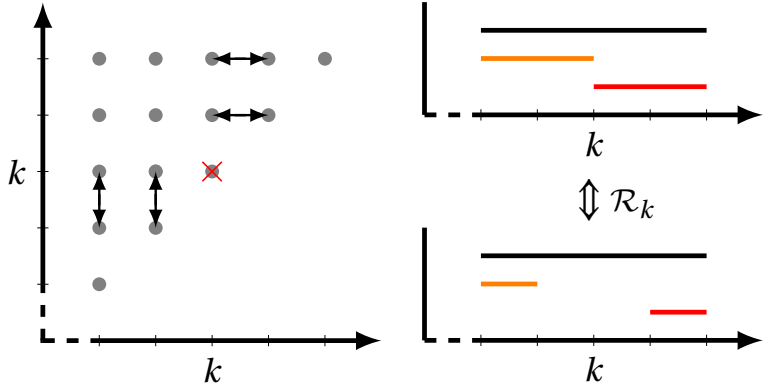

	\begin{example}
		Let $n = 4$, let $\tau = (\rightarrow,\leftarrow,\rightarrow)\in\altmathcal{T}_n$, and let $\pv$ be the $\tau$-module 
		\[\pv = \mathbb{I}_\tau([1,4])\oplus\mathbb{I}_\tau([1,2])\oplus\mathbb{I}_\tau([2,3])\oplus\mathbb{I}_\tau([2,3])\oplus\mathbb{I}_\tau([3,3])\]
		Evidently, $\tau$ has a sink at index $2$. Applying $\altmathcal{L}_2$ and using Theorems \ref{bigsumreflection} and \ref{intervalreflections}, we have
		\[\altmathcal{L}_2(\pv) \cong \mathbb{I}_{\sigma_2\tau}([1,4])\oplus\mathbb{I}_{\sigma_2\tau}([1,1])\oplus \mathbb{I}_{\sigma_2\tau}([3,3])^3.\]
		Similarly, we see that $\tau$ has a source at index $3$, and
		\[\altmathcal{C}_3(\pv) \cong \mathbb{I}_{\zeta_3\tau}([1,4])\oplus\mathbb{I}_{\zeta_3\tau}([1,2])\oplus \mathbb{I}_{\zeta_3\tau}([2,2])^2.\]
	\end{example}
	
	The following theorem is the analogue of Theorem \ref{intervalreflections} for interval modules having a flow, i.e. neither a sink or a source, at the index to which a reflection functor is applied:
	\begin{theorem}[The Diamond Principle Part II]\label{flow_reflection}
		Let $k\in\{2,\dots,n-1\}$ and let $\tau\in\altmathcal{T}_n$ have a forward flow at index $k$.
		\begin{enumerate}[(1.)]
			
			\item 
			The extroversion functor $\altmathcal{L}_k$ acts as follows on the interval $\tau$-modules:	
			\begin{equation*}
			\begin{aligned}
			\mathbb{I}_\tau([k,k])\, &&\stackrel{\altmathcal{L}_k }{\longrightarrow} & \quad \mathdszero &\\
			\mathbb{I}_\tau([b,k-1])\, &&\longrightarrow & \quad\mathbb{I}_{\sigma_k\tau}([b,k]) && \quad \text{for }b\leq k-1\\
			\mathbb{I}_\tau([k,d])\, &&\longrightarrow & \quad\mathbb{I}_{\sigma_k\tau}([k+1,d]) && \quad \text{for }d\geq k+1\\
			\mathbb{I}_\tau([b,d])\, &&\longrightarrow & \quad\mathbb{I}_{\sigma_k\tau}([b,d]) && \quad \text{otherwise}.\\
			\end{aligned}
			\end{equation*}

			\item
			The introversion functor $\altmathcal{C}_k$ acts as follows on the interval $\tau$-modules:	
			\begin{equation*}
			\begin{aligned}
			\mathbb{I}_\tau([k,k])\, &&\stackrel{\altmathcal{C}_k}{\longrightarrow} & \quad \mathdszero &\\
			\mathbb{I}_\tau([b,k])\, &&\longrightarrow & \quad\mathbb{I}_{\xi_k\tau}([b,k-1]) && \quad \text{for }b\leq k-1\\
			\mathbb{I}_\tau([k+1,d])\, &&\longrightarrow & \quad\mathbb{I}_{\xi_k\tau}([k,d]) && \quad \text{for }d\geq k+1\\
			\mathbb{I}_\tau([b,d])\, &&\longrightarrow & \quad\mathbb{I}_{\xi_k\tau}([b,d]) && \quad \text{otherwise}.\\
			\end{aligned}
			\end{equation*}
			
		\end{enumerate}
		\begin{proof} The proof is just a straightforward verification of the universal properties of limits and colimits. We omit the details.	\end{proof}
	\end{theorem}
	
	The same computations can be made for reflections applied to the indices with backwards flows or to the indices $1$ or $n$, though these details are not so important for us. The upshot of Theorems \ref{intervalreflections} and \ref{flow_reflection} is that points in the diagram of zigzag module move at furthest horizontally or vertically to adjacent nodes on the integer lattice (see Figures \ref{reflection_diagram} and \ref{inex_diagram}).
	
	\begin{figure}
		\centering
		\scalebox{1.44}{
			\begin{tikzpicture}
			\draw[-latex,very thick] (1,0.5)--(3.5,0.5);\draw[-latex,very thick] (0.5,1)--(0.5,3.5);
			
			\draw[dashed,very thick] (0.5,0.5)--(1,0.5);
			\draw[dashed,very thick] (0.5,0.5)--(0.5,1);
			\draw[very thick] (0.479,0.5)--(0.5,0.5);
			
			\fill[gray]  (1.5,2) circle [radius=1.75pt];
			\fill[gray]  (1,2.5) circle [radius=1.75pt];
			\fill[gray]  (1.5,2) circle [radius=1.75pt];\fill[gray]  (2,2) circle [radius=1.75pt];
			\fill[gray]  (1,1.5) circle [radius=1.75pt];\fill[gray]  (1,2) circle [radius=1.75pt];
			\fill[gray]  (1.5,1.5) circle [radius=1.75pt];\fill[gray]  (1,1) circle [radius=1.75pt];
			\fill[gray]  (2.5,3) circle [radius=1.75pt];\fill[gray]  (2,3) circle [radius=1.75pt];
			\fill[gray]  (1.5,3) circle [radius=1.75pt];\fill[gray]  (1,3) circle [radius=1.75pt];
			\fill[gray] (2.5,2.5) circle [radius=1.75pt];\fill[gray] (2,2.5) circle [radius=1.75pt];
			\fill[gray]  (1.5,2.5) circle [radius=1.75pt];\fill[gray]  (3,3) circle [radius=1.75pt];
			
			\tikzset{cross/.style={cross out, draw=red, fill=none, minimum size=3, inner sep=0pt, outer sep=0pt}, cross/.default={2pt}}
			
			\node[draw,cross out,red,inner sep = 1pt,minimum size = 4.5pt] at (2,2){};

			\draw[-latex,semithick] (1.5,1.5)--(1.5,2);
			
			\draw[-latex,semithick] (1.0,1.5)--(1.0,2);
			
			\draw[-latex,semithick] (2.0,2.5)--(2.5,2.5);
			
			\draw[-latex,semithick] (2.0,3.0)--(2.5,3.0);

			\footnotesize
			\draw[very thin] (2,0.55) -- (2,0.45)node[anchor=north ] {$k$};
			\draw[very thin] (2.5,0.55) -- (2.5,0.45);
			\draw[very thin] (3,0.45) -- (3,0.55);
			\draw[very thin] (1,0.55) -- (1,0.45);
			\draw[very thin] (1.5,0.55) -- (1.5,0.45);
			
			\footnotesize
			\draw[very thin] (.55,2) -- (.45,2)node[anchor=east ] {$k$};
			\draw[very thin] (.55,2.5) -- (.45,2.5);
			\draw[very thin] (.55,3) -- (0.45,3);
			\draw[very thin] (.55,1) -- (0.45,1);
			\draw[very thin] (.55,1.5) -- (.45,1.5);
			\node [anchor=east] at (2.75,1.25) {$\altmathcal{L}_k$};
			\end{tikzpicture}\quad
			
			\begin{tikzpicture}
			\draw[-latex,very thick] (1,0.5)--(3.5,0.5);\draw[-latex,very thick] (0.5,1)--(0.5,3.5);
			
			\draw[dashed,very thick] (0.5,0.5)--(1,0.5);
			\draw[dashed,very thick] (0.5,0.5)--(0.5,1);
			\draw[very thick] (0.479,0.5)--(0.5,0.5);
			
			\fill[gray]  (1.5,2) circle [radius=1.75pt];
			\fill[gray]  (1,2.5) circle [radius=1.75pt];
			\fill[gray]  (1.5,2) circle [radius=1.75pt];\fill[gray]  (2,2) circle [radius=1.75pt];
			\fill[gray]  (1,1.5) circle [radius=1.75pt];\fill[gray]  (1,2) circle [radius=1.75pt];
			\fill[gray]  (1.5,1.5) circle [radius=1.75pt];\fill[gray]  (1,1) circle [radius=1.75pt];
			\fill[gray]  (2.5,3) circle [radius=1.75pt];\fill[gray]  (2,3) circle [radius=1.75pt];
			\fill[gray]  (1.5,3) circle [radius=1.75pt];\fill[gray]  (1,3) circle [radius=1.75pt];
			\fill[gray] (2.5,2.5) circle [radius=1.75pt];\fill[gray] (2,2.5) circle [radius=1.75pt];
			\fill[gray]  (1.5,2.5) circle [radius=1.75pt];\fill[gray]  (3,3) circle [radius=1.75pt];
			
			\tikzset{cross/.style={cross out, draw=red, fill=none, minimum size=3, inner sep=0pt, outer sep=0pt}, cross/.default={2pt}}
			
			\node[draw,cross out,red,inner sep = 1pt,minimum size = 4.5pt] at (2,2){};
			
			\draw[-latex,semithick] (1.5,2.0)--(1.5,1.5);
			
			\draw[-latex,semithick] (1.0,2.0)--(1.0,1.5);
			
			\draw[-latex,semithick] (2.5,2.5)--(2,2.5);
			
			\draw[-latex,semithick] (2.5,3.0)--(2,3.0);
			
			\footnotesize
			\draw[very thin] (2,0.55) -- (2,0.45)node[anchor=north ] {$k$};
			\draw[very thin] (2.5,0.55) -- (2.5,0.45);
			\draw[very thin] (3,0.45) -- (3,0.55);
			\draw[very thin] (1,0.55) -- (1,0.45);
			\draw[very thin] (1.5,0.55) -- (1.5,0.45);
			
			\footnotesize
			\draw[very thin] (.55,2) -- (.45,2)node[anchor=east ] {$k$};
			\draw[very thin] (.55,2.5) -- (.45,2.5);
			\draw[very thin] (.55,3) -- (0.45,3);
			\draw[very thin] (.55,1) -- (0.45,1);
			\draw[very thin] (.55,1.5) -- (.45,1.5);
			\node [anchor=east] at (2.75,1.25) {$\altmathcal{C}_k$};
			\end{tikzpicture}\quad}
		\caption{Points in the persistence diagram of a zigzag module $\mathds{V}$ with a \textit{forward flow} at index $k$ move according to the arrows when the extroversion functor $\altmathcal{L}_k$ (left) or introversion functor $\altmathcal{C}_k$ (right) are applied. The point $(k,k)$ corresponding to the simple summand $\mathbb{I}([k,k])$ is annihilated by both of these functors.}
		\label{inex_diagram}        
	\end{figure}
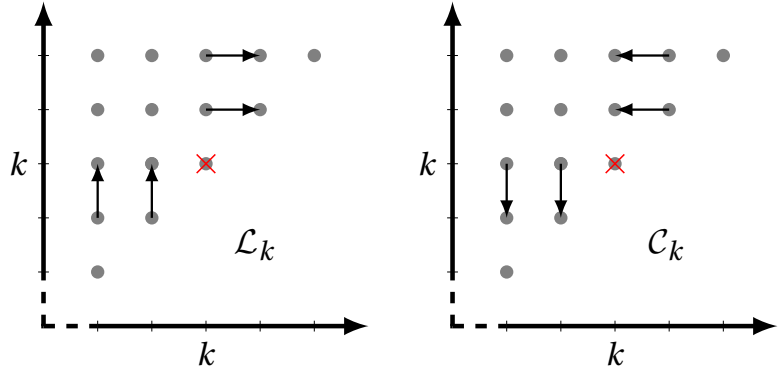
	
	\subsection{Removing Simple Summands}\label{sec:sanitation}
	In the next section we define the reflection distance. In order to avoid having the distance between a simple summand and the zero zigzag module be nonzero, we introduce a map sending a zigzag module $\pv$ to a zigzag module $\pv'$ obtained from $\pv$ by removing simple summands.
	
	We define the map $\altmathcal{S}:\nmod\to\nmod$ as follows: given a zigzag module $\pv$ with interval decomposition $\pv = \bigoplus_{[b,d]\in \dgm(\pv)}\mathbb{I}([b,d])$, we define
	\[\altmathcal{S}(\pv):= \bigoplus_{\substack{[b,d]\in \dgm(\pv)\\b\neq d}}\mathbb{I}([b,d]).\]
	That is, $\altmathcal{S}(\pv)$ is obtained from $\pv$ by removing simple summands. If $\pv$ is a direct sum of simple modules, then we set $\altmathcal{S}(\pv) := \mathdszero$. We also set $\altmathcal{S}(\mathdszero) := \mathdszero$. By definition, $\altmathcal{S}$ sends zigzag modules of type $\tau$ to zigzag modules of type $\tau$ so that $\altmathcal{S}$ restricts to a map $\altmathcal{S}:\taumod{\tau}\to\taumod{\tau}$ for each type $\tau$.
	\begin{remarks}\label{sanitize_alternate_def}
		\begin{enumerate}[(1)]
			\item If  $\,\pv = \pv_1\oplus\pv_2$ where $\pv_1$ is a direct sum of simple modules and $\pv_2$ has no simple summands, then $\altmathcal{S}(\pv) = \pv_2$.
			\item By definition, $\altmathcal{S}(\pv)\preceq \pv$ and hence $\altmathcal{S}(\pv)\precsim\pv$ for any zigzag module $\pv$.
			\item $\altmathcal{S}(\altmathcal{S}(\pv)) = \altmathcal{S}(\pv)$ for all zigzag modules $\pv$.
			\item $\dgm(\altmathcal{S}(\pv)) = \lms(b,d)\in \dgm(\pv) \ | \ b\neq d\rms$. That is, the diagram of $\altmathcal{S}(\pv)$ is that of $\,\pv$ minus the points on the diagonal.
		\end{enumerate}
	\end{remarks}
	
	We also have the following analogues of Corollaries \ref{refsummand} and \ref{refprecsim}:
	
	\begin{proposition}\label{sanitize_summand}
		Fix $k\in\{1,\hdots,n\}$, let $\tau\in \altmathcal{T}_n$, and let $\pv$, $\pw\in\taumod{\tau}$. If $\,\pw\preceq \pv$ then $\altmathcal{S}(\pw)\preceq\altmathcal{S}(\pv)$.
		\begin{proof}If $\pw\preceq\pv$ then $\pv \cong \pw\oplus\pu$ for some $\pu\in\taumod{\tau}$. Write
			\[\pw= \pw_1\oplus \pw_2 \quad \textup{ and } \quad \pu = \pu_1\oplus\pu_2,\]
			where $\pw_1,\pu_1$ are direct sums of simple modules and $\pw_2,\pu_2$ have no simple summands. By Remark \ref{sanitize_alternate_def}, $\altmathcal{S}(\pw) = \pw_2$ and $\altmathcal{S}(\pu) = \pu_2$. Moreover $\pv \cong (\pw_1\oplus\pu_1)\oplus(\pw_2\oplus\pu_2)$, with $\pw_1\oplus\pu_1$ being a direct sum of simple modules and $\pw_2\oplus\pu_2$ having no simple summands. Then again by Remark \ref{sanitize_alternate_def},
			\[ \altmathcal{S}(\pv) \cong \pw_2\oplus\pu_2 = \altmathcal{S}(\pw)\oplus \altmathcal{S}(\pu),\]
			and hence $\altmathcal{S}(\pw)\preceq \altmathcal{S}(\pv)$.
		\end{proof}
	\end{proposition}
	
	\begin{proposition}\label{sanitize_precsim}
		Fix $k\in\{1,\hdots,n\}$, let $\tau\in \altmathcal{T}_n$, and let $\pv$, $\pw\in\nmod$. If $\,\pw\precsim \pv$ then $\altmathcal{S}(\pw)\precsim\altmathcal{S}(\pv)$.
		\begin{proof}
			If $\pw\precsim \pv$ then $\pw\sim\pw'\preceq \pv$ for some $\pw'$. Now note that that if $\pw\sim \pw'$ then $\altmathcal{S}(\pw)\sim \altmathcal{S}(\pw')$. This is because if an arrow representing an isomorphism is reversed, then the source and target of this arrow cannot be an index at which a simple summand appears. From this fact and Corollary \ref{sanitize_summand}, we have \[\altmathcal{S}(\pw)\sim\altmathcal{S}(\pw')\preceq \altmathcal{S}(\pv),\]
			and hence $\altmathcal{S}(\pw)\precsim\altmathcal{S}(\pv)$.
		\end{proof}
	\end{proposition}
	
	\begin{definition}
		Define an equivalence relation $\sim_\altmathcal{S}$ on $\nmod$ by
		\[ \pv\sim_\altmathcal{S}\pw \quad \iff \quad \altmathcal{S}(\pv)\precsim\pw \quad \textup{ and } \quad \altmathcal{S}(\pw)\precsim\pv.\]
	\end{definition}
	\noindent Reflexivity of $\sim_{\altmathcal{S}}$ follows from the fact that $\altmathcal{S}(\pv)\precsim \pv$ and symmetry is immediate from the definition. To see that $\sim_{\altmathcal{S}}$ is transitive, suppose that $\pv\sim_\altmathcal{S}\pw$ and $\pw \sim_{\altmathcal{S}}\pu$. Then $$\altmathcal{S}(\pv)\precsim \pw,\quad  \altmathcal{S}(\pw)\precsim \pv,\quad  \altmathcal{S}(\pw)\precsim \pu, \quad \textup{ and }\quad \altmathcal{S}(\pu)\precsim \pw.$$
	Using Proposition \ref{sanitize_precsim} together with the fact that $\altmathcal{S}(\altmathcal{S}(\pv)) = \altmathcal{S}(\pv)$ for all zigzag modules $\pv$, we have
	
	\[\altmathcal{S}(\pv) = \altmathcal{S}(\altmathcal{S}(\pv))\precsim \altmathcal{S}(\pw)\precsim\pu \quad \textup{ and } \quad \altmathcal{S}(\pu) = \altmathcal{S}(\altmathcal{S}(\pu))\precsim \altmathcal{S}(\pw)\precsim\pv.
	\]
	By transitivity of the preorder $\precsim$, we have $\altmathcal{S}(\pv)\precsim \pu$ and $\altmathcal{S}(\pu)\precsim \pv$ so that $\pv \sim_{\altmathcal{S}}\pu$. Thus $\sim_{\altmathcal{S}}$ is indeed an equivalence relation on $\nmod$.

	\section{The Reflection Distance}\label{sec:ref-distance}
	
	Let $\mathfrak{R} = (\altmathcal{R}_{k_\ell},\altmathcal{R}_{k_{\ell-1}},\cdots,\altmathcal{R}_{k_1})$ be a sequence of $\ell$ reflection functors with $1\leq k_i\leq n$ for all $i$. For $\pv\in\nmod$, we denote by $\mathfrak{R}(\pv)$ the zigzag module \[(\altmathcal{S}\circ\altmathcal{R}_{k_\ell}\circ\altmathcal{S}\circ\altmathcal{R}_{k_{\ell-1}}\circ\cdots \circ\altmathcal{S}\circ\altmathcal{R}_{k_1}\circ\altmathcal{S})(\pv),\] 
	where $\altmathcal{S}$ is the map defined in Section \ref{sec:sanitation}. That is, $\mathfrak{R}(\pv)$ is obtained from $\pv$ by first removing simple summands from $\pv$, applying the reflection $\altmathcal{R}_{k_1}$, removing simple summands from the resulting zigzag module, and so on.  We also set $\varepsilon(\pv) :=\altmathcal{S}(\pv)$, where $\varepsilon$ denotes the empty sequence. 
	
	For each $p\in[1,\infty)$ define the \textit{$p$-cost} of a nonempty sequence $\mathfrak{R}$ of length $\ell$ by
	\[C_p(\mathfrak{R}):= \ell^{1/p}\]
	and define $C_p(\varepsilon) = 0$.

	\begin{definition}\label{reflection_distance} For each $p\in[1,\infty)$, we define a function $d_{\altmathcal{R}}^p:\nmod\times\nmod\to \R$ by setting
		\[ d_{\altmathcal{R}}^p(\pv,\pw):= \min_{(\mathfrak{R},\mathfrak{R}')}\big\{\max\{C_p(\mathfrak{R}),C_p(\mathfrak{R}')\} \ | \ \mathfrak{R}(\pv)\precsim \pw \text{ and } \mathfrak{R}'(\pw)\precsim \pv\big\},\]
		where the minimum is taken over all pairs $(\mathfrak{R},\mathfrak{R}')$ of sequences of reflection functors.
	\end{definition}
	
	Given two sequences $\mathfrak{R} = (\altmathcal{R}_{k_\ell},\altmathcal{R}_{k_{\ell-1}},\dots,\altmathcal{R}_{k_1})$, $\mathfrak{R}' = (\altmathcal{R}_{m_j},\altmathcal{R}_{m_{j-1}},\dots,\altmathcal{R}_{m_1})$ of reflection functors, denote by $\mathfrak{R}'\circ \mathfrak{R}$ the concatenation $(\altmathcal{R}_{m_j},\dots,\altmathcal{R}_{m_1},\altmathcal{R}_{k_\ell},\dots,\altmathcal{R}_{k_1})$. 
	\begin{fact}\label{costadd}
		For any $p\in [1,\infty)$ we have $C_p(\mathfrak{R}' \circ \mathfrak{R})\leq C_p(\mathfrak{R}) + C_p(\mathfrak{R}')$.
		\begin{proof} This follows from the inequality $(a + b)^{1/p}\leq a^{1/p} + b^{1/p}$ for all nonnegative integers $a,b$.
		\end{proof}
	\end{fact}
	
	We will say that a sequence $\mathfrak{R}$ of reflection functors \textit{annihilates} a summand $\pw$ of $\pv$ if $\mathfrak{R}(\pw) = \mathdszero$. Note that if $\pv = \pw\oplus \pu$ and $\mathfrak{R}$ annihilates $\pw$ then $\mathfrak{R}(\pv) \cong \mathfrak{R}(\pu)$.
	
	\begin{proposition}
		For any $\pv\in\nmod$ there exists a sequence of reflection functors $\mathfrak{R}$ which annihilates $\pv$, i.e., $\mathfrak{R}(\pv) = \mathdszero$.
		\begin{proof}
			Let $\tau = \textup{type}(\pv)$. It suffices to show that for any interval module $\mathbb{I}_\tau([b,d])$, there is a composition of reflections $\mathfrak{R}$ such that $\mathfrak{R}(\mathbb{I}_\tau([b,d])) = \mathdszero$. For if this is the case then we can iteratively annihilate interval summands of $\pv$ until we arrive at $\mathdszero$.
			
			We demonstrate that this can be done as follows: the interval $\tau$-module $\mathbb{I}_\tau([b,d])$ has either a sink, source, or flow at index $d$. In any case, Theorems \ref{intervalreflections} and \ref{flow_reflection} guarantee that we can choose a reflection functor $\altmathcal{R}_d$ at index $d$ such that $\altmathcal{R}_d(\mathbb{I}([b,d]))$ is an interval module supported over $[b,d-1]$. In this way, we obtain a sequence $\mathfrak{R}_1  = (\altmathcal{R}_{b},\altmathcal{R}_{b+1},\cdots,\altmathcal{R}_{d-1},\altmathcal{R}_{d})$ such that $\mathfrak{R}_1(\mathbb{I}_\tau([b,d]))= \mathdszero$. Applying this procedure again to an interval summand of $\mathfrak{R}^1(\pv)$ yields a sequence $\mathfrak{R}_2$ of reflection functors which annihilates this summand. Since $\pv\in \nmod$ has only finitely many summands, and since applying a reflection functor to a zigzag module can only reduce the number of summands, we obtain a sequence $(\mathfrak{R}_1,\dots,\mathfrak{R}_m)$ of sequences of reflection functors, where $m \leq|\dgm(\pv)|$, for which the concatenation $\mathfrak{R} = \mathfrak{R}_m\circ\cdots\circ\mathfrak{R}_1$ satisfies $\mathfrak{R}(\pv) = \mathdszero$.
		\end{proof}
	\end{proposition}
	
	\begin{remark}\label{finiteness_remark}
		By Theorems \ref{intervalreflections} and \ref{flow_reflection}, it is clear that it takes exactly $d-b$ reflections to annihilate an interval module supported over $[b,d]$. It follows that $d_{\altmathcal{R}}^p(\mathbb{I}([b,d]),\mathdszero) = (d-b)^{1/p}$. Moreover, if $|\dgm(\pv)| = K$ for $\pv\in \nmod$ then $d_{\altmathcal{R}}^p(\pv,\mathdszero) \leq Kn$. This follows since we can iteratively annihilate each of the $K$ interval summands of $\,\pv$ with at most $n$ reflections.
		
	\end{remark}

	\begin{theorem}\label{thm:ref_is_metric}
		For each $n\in \N$ and for all $p\in[1,\infty)$, the function $d_{\altmathcal{R}}^p$ is a pseudometric on $n\textup{-Mod}$. Moreover, $d_{\altmathcal{R}}^p(\pv,\pw) = 0$ if and only if $\,\pv\sim_{\altmathcal{S}}\pw$ so that $d_{\altmathcal{R}}^p$ induces a metric on $n\textup{-Mod}/{\sim_\altmathcal{S}}$.
		\begin{proof} It follows from Remark \ref{finiteness_remark} that $d_\altmathcal{R}^p(\pv,\pw)$ is finite for all $\pv,\pw\in \nmod$. The facts that $d_{\altmathcal{R}}^p$ is a non-negative and symmetric function follow immediately. Since $\epsilon(\pv) = \altmathcal{S}(\pv)\precsim\pv$ we have $d_{\altmathcal{R}}^p(\pv,\pv) = 0$ for all $\pv\in n\textup{-Mod}$.
			
			Next, we verify the triangle inequality. Fix $p\in [1,\infty)$ and let $\pv_1$, $\pv_2$, $\pv_3\in\nmod$. Then there exist sequences of reflections $\mathfrak{R}_1, \mathfrak{R}_2$ such that           
			\[\mathfrak{R}_1(\pv_1)\precsim\pv_2,\quad \mathfrak{R}_2(\pv_2)\precsim\pv_1 \quad \textup{ with } \quad C_p(\mathfrak{R}_1)\leq d_{\altmathcal{R}}^p(\pv_1,\pv_2), \quad C_p(\mathfrak{R}_2)\leq d_{\altmathcal{R}}^p(\pv_1,\pv_2).
			\]
			Similarly, there exist sequences of reflections $\mathfrak{R}_1', \mathfrak{R}_2'$ such that 
			\[\mathfrak{R}_1'(\pv_2)\precsim\pv_3, \quad \mathfrak{R}_2'(\pv_3)\precsim\pv_2\quad \textup{ with } \quad C_p(\mathfrak{R}_1')\leq d_{\altmathcal{R}}^p(\pv_2,\pv_3), \quad C_p(\mathfrak{R}_2')\leq d_{\altmathcal{R}}^p(\pv_2,\pv_3).\]
			Then by Corollary \ref{refprecsim} and Proposition \ref{sanitize_precsim}, $\mathfrak{R}_1'\circ\mathfrak{R}_1(\pv_1)\precsim\mathfrak{R}_1'(\pv_2)\precsim \pv_3$ and $\mathfrak{R}_2\circ\mathfrak{R}_2'(\pv_3)\precsim\mathfrak{R}_2(\pv_2)\precsim\pv_1$. By transitivity of the preorder $\precsim$, we have
			\begin{equation}\label{composite_ref}\mathfrak{R}_1'\circ\mathfrak{R}_1(\pv_1)\precsim \pv_3\quad \textup{ and }\quad \mathfrak{R}_2\circ\mathfrak{R}_2'(\pv_3)\precsim\pv_1.\end{equation}
			By Fact \ref{costadd}, $C_p(\mathfrak{R}_1'\circ\mathfrak{R}_1)\leq C_p(\mathfrak{R}_1')+C_p(\mathfrak{R}_1)$ and $C_p(\mathfrak{R}_2\circ\mathfrak{R}_2')\leq C_p(\mathfrak{R}_2) + C_p(\mathfrak{R}_2')$ so that
			\begin{equation} \label{composite_cost}C_p(\mathfrak{R}_1'\circ\mathfrak{R}_1)\leq d_{\altmathcal{R}}^p(\pv_1,\pv_2) + d_{\altmathcal{R}}^p(\pv_2,\pv_3)\quad \textup{ and }\quad C_p(\mathfrak{R}_2\circ\mathfrak{R}_2')\leq d_{\altmathcal{R}}^p(\pv_1,\pv_2) + d_{\altmathcal{R}}^p(\pv_2,\pv_3).\end{equation} 
			Equations (\ref{composite_ref}) and (\ref{composite_cost}) together imply \[d_{\altmathcal{R}}^p(\pv_1,\pv_3)\leq  d_{\altmathcal{R}}^p(\pv_1,\pv_2) + d_{\altmathcal{R}}^p(\pv_2,\pv_3).\]

			Now if $\pv\sim_{\altmathcal{S}}\pw$ then $\epsilon(\pv) = \altmathcal{S}(\pv)\precsim\pw$ and $\epsilon(\pw) = \altmathcal{S}(\pw)\precsim \pv$ so that $d_{\altmathcal{R}}^p(\pv,\pw) = 0$. Conversely, if $\pv \not\sim_\altmathcal{S} \pw$ then either $\altmathcal{S}(\pv)\not\precsim\pw$ or $\altmathcal{S}(\pw)\not\precsim\pv$. In any case, we must apply some non-trivial reflection to either $\pv$ or $\pw$, incurring a nonzero cost so that $d_{\altmathcal{R}}^p(\pv,\pw) > 0$. It follows that if $\pv \sim_\altmathcal{S}\pv'$ and $\pw\sim_\altmathcal{S} \pw'$ then $d_{\altmathcal{R}}^p(\pv,\pw) = d_{\altmathcal{R}}^p(\pv',\pw')$ so that $d_{\altmathcal{R}}^p$ induces a well-defined metric on the equivalence classes of $\nmod$ under $\sim_\altmathcal{S}$.
		\end{proof}
	\end{theorem}

	\begin{corollary}\label{corollary:distance_is_zero}
		If $\,\pv\in \nmod$ is a direct sum of simple modules then $d_{\altmathcal{R}}^p(\pv,\mathdszero) = 0$.
		\begin{proof}
			We have $\altmathcal{S}(\pv) = \mathdszero \precsim \mathdszero$ and $\altmathcal{S}(\mathdszero) = \mathdszero\precsim \pv$ so that $\pv\sim_{\altmathcal{S}}\mathdszero$ and hence $d_{\altmathcal{R}}^p(\pv,\mathdszero) = 0$ by the second statement of Theorem \ref{thm:ref_is_metric}.
		\end{proof}
	\end{corollary}
	
	\section{The $\ell^p$-Bottleneck Distance and the Proof of the Main Theorem}\label{sec:d-bottleneck}
	In this section we define a family of metrics on the space of persistence diagrams which we call the $\ell^p$-bottleneck distances. We show that the map which takes a zigzag module to its persistence diagram is 1-Lipschitz with respect to the 1-reflection distance and the $\ell^1$-bottleneck distance.
	\subsection{Matchings}
	Let $S$ and $T$ be sets. A \textit{matching} between $S$ and $T$ is a relation $M\subseteq S\times T$ such that 
	\begin{enumerate}[(1)]
		\item For any $s\in S$, there is at most one $t\in T$ such that $(s,t)\in M$,
		\item For any $t\in T$, there is a most one $s\in S$ such that $(s,t)\in M$.
	\end{enumerate}
	We denote a matching by $M:S\nrightarrow T$.
	Equivalently, a matching is a bijection $M:S'\to T'$ for some subsets $S'\subseteq S$ and $T'	\subseteq T$. From this point of view, $S'$ is called the \textit{coimage} of $M$, denoted $S' = \coim(M)$, and $T'$ is called the \textit{image} of $M$, denoted $T' = \im(M)$. A matching $M$ is said to be \textit{finite} if $|\coim(M)| = |\im(M)|$ is finite. If $(s,t)\in M$ then $s$ and $t$ are said to be \textit{matched}. Points in $S\sqcup T$ which are not matched are said to be \textit{unmatched}.
	
	The following classical theorem from matching theory is crucial to our proof of stability:

	\begin{lemma}[Matching Lemma, \cite{ore1962theory}]\label{matching_lemma}
		Let $S$ and $T$ be sets and let $f:S\nrightarrow T$ and $g:T\nrightarrow S$ be matchings. Then there exists a matching $M:S\nrightarrow T$ such that 
		\begin{enumerate}[label=(\arabic*)]
			\item $\coim(f)\subseteq\coim(M)$,
			\item $ \coim(g)\subseteq\im(M)$,
			\item if $M(s) = t$ then either $f(s) = t$ or $g(t) = s$.
		\end{enumerate} 	
	\end{lemma} \noindent The Matching Lemma holds for arbitrary matchings but, for our purposes, we will only use the Matching Lemma in the case of finite matchings. A proof of the Matching Lemma \ref{matching_lemma} for the case of finite matchings is provided in Appendix \ref{app:matchings}.
	
	\subsection{The $\ell^p$-Bottleneck Distance} Let $S$ and $T$ be multisets of points from $\R^2$ and let $M:S\nrightarrow T$ be a matching. For each $p\in [1,\infty]$ define
	\begin{equation}\label{matching_cost} c_p(M) := \max\left\{\max_{(s,t)\in M}\|s-t\|_{p},\max_{r\in S\sqcup T \textup{ unmatched}}\frac{|r_y-r_x|}{2^{1-1/p}}\right\},\end{equation}
	where $r_x$ and $r_y$ denote the $x$ and $y$ coordinates, respectively, of the point $r\in \R^2$, and where $\|\cdot\|_p:\R^2\to \R$ denotes the usual $\ell^p$-norm on $\R^2$. Here we use the convention $1/\infty = 0$. We then define the \textit{$\ell^p$-bottleneck distance} between $S$ and $T$ by
	\[d_b^p(S,T) : = \inf_{M:S\nrightarrow T}c_p(M),\]
	where the infimum is taken over all matchings between $S$ and $T$.
	
	\begin{remark}\label{remark:db1_bottleneck_equiv}
		The bottleneck distance (as defined in \cite{bauer2014induced}) is just a special case of the $\ell^p$-bottleneck distance when $p = \infty$. It is simply the limit of $d_b^p$ as $p\to \infty$. Since $\|x\|_\infty\leq\|x\|_1\leq 2\|x\|_\infty$ for all $x\in \R^2$, we see that $d_b^\infty(S,T)\leq d_b^1(S,T)\leq 2d_b^\infty(S,T)$ for all multisets $S$ and $T$.
	\end{remark}
	
	Given a multisubset $S$ of points in $\R^2$, and given $p\in [1,\infty]$ and $\eta>0$, we define
	\[S^\eta_p := \left\{\!\!\!\left\{ s\in S \ \Big| \ \frac{|s_y-s_x|}{2^{1-1/p}} >\eta \right\}\!\!\!\right\}.\]
	
	\begin{lemma}\label{matching_lemma_two}
		Fix $p\in[1,\infty]$, let $S$ and $T$ be multisubsets of points in $\R^2$, and let $M:S\nrightarrow T$ be a matching such that 
		\begin{enumerate}[(1)]
			\item $S^\eta_p \subseteq \coim(M)$,
			\item $T^\eta_p \subseteq \im(M)$,
			\item if $M(s) = t$ then $\|s-t\|_p\leq \eta$.
			
		\end{enumerate}
		Then $c_p(M)\leq \eta$.
		\begin{proof}
			Condition (3) guarantees that $\max_{(s,t)\in M}\|s-t\|_p\leq \eta$. If $r\in S\sqcup T$ is unmatched then $r\not\in S^\eta_p\sqcup T^\eta_p$ so that $\frac{|r_y-r_x|}{2^{1-1/p}}\leq \eta$. That $c_p(M)\leq \eta$ now follows from equation (\ref{matching_cost}).
		\end{proof}
	\end{lemma}

	\subsection{Stability of Persistence Diagrams with Respect to $d_\altmathcal{R}^1$}
	In this section, we show that the $\ell^1$-bottleneck distance between the persistence diagrams of two given zigzag modules is bounded above by the 1-reflection distance between the zigzag modules themselves.
	\begin{lemma}\label{cost_to_kill}
		Let $\tau\in\altmathcal{T}_n$ and $\pv\in\taumod{\tau}$. If  $\,\mathfrak{R}(\pv) = \mathdszero$ then 
		\[C_1(\mathfrak{R})\geq \displaystyle\max_{(b,d)\in\dgm(\pv)}(d-b).\] Moreover, if $\,\mathfrak{R}(\mathbb{I}([b,d]))=\mathbb{I}([b',d'])$ then \[C_1(\mathfrak{R})\geq |d'-d| + |b'-b|.\]
	\end{lemma}
	\begin{proof}
		By Theorem \ref{bigsumreflection} we have \[\mathdszero  = \mathfrak{R}(\pv)  = \mathfrak{R}\left(\bigoplus_{(b,d)\in\dgm(\pv)}\mathbb{I}_{\tau}([b,d])\right) \cong \bigoplus_{(b,d)\in\dgm(\pv)}\mathfrak{R}(\mathbb{I}_{\tau}([b,d]))\]
		so that $\mathfrak{R}(\mathbb{I}_{\tau}([b,d])) = \mathdszero$ for all $(b,d)\in\dgm(\pv)$. Hence $C_1(\mathfrak{R})\geq d-b$ for all $(b,d)\in \dgm(\pv)$, proving the first claim.
		%
		
		To prove the second claim, suppose that $\mathfrak{R}$ is a sequence of $m$ reflections so that $C_1(\mathfrak{R}) = m$. By Theorems \ref{intervalreflections} and \ref{flow_reflection}, any reflection moves points in corresponding diagram a distance of at most $1$ in the $\ell^1$ norm. Thus if $\mathfrak{R}(\mathbb{I}([b,d]))=\mathbb{I}([b',d'])$ then $(b',d')$ lives in the $m$-ball of $(b,d)$ with respect to the $\ell^1$ norm. Hence 
		\[ \|(b,d)-(b',d')\|_1 = |d'-d| + |b'-b| \leq m = C_1(\mathfrak{R}).\qedhere\]
	\end{proof}
	We can now prove our main theorem.

	\thmmain*

	\begin{proof}
		We show that for any pair of sequences of reflections $(\mathfrak{R}^1,\mathfrak{R}^2)$ with $\mathfrak{R}^1(\pv)\precsim \pw$ and $\mathfrak{R}^2(\pw)\precsim \pv$, there is a matching $M:\dgm(\pw)\nrightarrow \dgm(\pv)$ satisfying $c_1(M) \leq \max\{C_1(\mathfrak{R}^1),C_1(\mathfrak{R}^2)\}$.
		
		Let $\eta = \max\{C_1(\mathfrak{R}^1),C_1(\mathfrak{R}^2)\}$. Consider the multisubsets
		\[\altmathcal{V}_1^\eta = \lms(b,d)\in\dgm(\pv) \ | \ |d-b|>\eta \rms,\] \[\altmathcal{W}_1^\eta = \lms(b',d')\in\dgm(\pw) \ | \ |d'-b'| > \eta\rms\]
		and
		\[\altmathcal{I} = \lms(b,d)\in \dgm(\pv) \ | \ \mathfrak{R}^1(\mathbb{I}([b,d])) \neq \mathdszero \rms,\]
		\[ \altmathcal{J} = \lms(b',d')\in \dgm(\pw) \ | \ \mathfrak{R}^2(\mathbb{I}([b',d'])) \neq  \mathdszero \rms.\]
		Note that by Lemma \ref{cost_to_kill}, $\altmathcal{V}_1^\eta\subseteq \altmathcal{I}$ and $\altmathcal{W}_1^\eta \subseteq \altmathcal{J}$.

		Let $\alpha_1:\altmathcal{I}\to \dgm(\mathfrak{R}^1(\pv))$ and $\alpha_2:\altmathcal{J} \to \dgm(\mathfrak{R}^2(\pw))$ be the injections given by $\alpha_i(b,d) = (b',d')$ if and only if $\mathfrak{R}^i(\mathbb{I}[b,d]) = \mathbb{I}[b',d']$ for each $i\in\{1,2\}$. By Theorem \ref{perssum} and Proposition \ref{simpers}, we have $\dgm(\mathfrak{R}^1(\pv))\subseteq \dgm(\pw)$ and $\dgm(\mathfrak{R}^2(\pw))\subseteq\dgm(\pv)$. Let $j_1$ and $j_2$ denote the respective inclusion maps. Then $M_1 := j_1\circ\alpha_1 :\altmathcal{I}\to \dgm(\pw)$ and $M_2 := j_2\circ\alpha_2 : \altmathcal{J}\to \dgm(\pv)$ are injective and hence can be viewed as matchings $M_1:\dgm(\pv)\nrightarrow\dgm(\pw)$ and $M_2:\dgm(\pw)\nrightarrow\dgm(\pv)$. Note that $\altmathcal{V}_1^\eta\subseteq \coim(M_1)$ and $\altmathcal{W}_1^\eta\subseteq \coim(M_2)$.

		Let $M:\dgm(\pv)\nrightarrow\dgm(\pw)$ be the matching constructed from $M_1$ and $M_2$ provided by the Matching Lemma \ref{matching_lemma}. This matching has the following properties:
		\begin{enumerate}[(1)]
			\item $\coim(M_1)\subseteq \coim(M)$,
			\item $\coim(M_2)\subseteq \im(M)$,
			\item if $M(b,d) = (b',d')$ then either $M_1(b,d) = (b',d')$ or $M_2(b,d) = (b',d')$.
		\end{enumerate}
		In particular, we have $\altmathcal{V}_1^\eta\subseteq \coim(M)$ and $\altmathcal{W}_1^\eta\subseteq \im(M)$. Also, if $M(b,d) = (b',d')$ then either $\mathfrak{R}^1(\mathbb{I}([b,d])) = \mathbb{I}([b',d'])$ or $\mathfrak{R}^2(\mathbb{I}([b,d])) = \mathbb{I}([b',d'])$. By Lemma \ref{cost_to_kill}, $\|(b,d)-(b',d')\|_1\leq \max\{C_1(\mathfrak{R}^1),C_2(\mathfrak{R}^2)\} = \eta$. Hence by Lemma \ref{matching_lemma_two}, $c_1(M)\leq \eta$, completing the proof.
	\end{proof}
	
	
	Since $d_b^\infty\leq d_b^1$, we have the following
	\begin{corollary}\label{bottleneckbound} For any $\pv,\pw\in\nmod$ we have
		$d_{b}^\infty(\dgm(\pv),\dgm(\pw)) \leq d_{\altmathcal{R}}^1(\pv,\pw).$
	\end{corollary}

	\subsection{Comparison with $d_b^1$ and the Bottleneck Distance}\label{sec:comparison_with_bottleneck}
	The next proposition shows that restricting the reflection distance to $\taumod{\tau}$ for a fixed type $\tau$, we obtain bi-Lipschitz equivalence of $d_{\altmathcal{R}}^1$ and $d_b^1$. First, we need the following lemmas:
	
	\begin{lemma}\label{lem:bound_to_zero}
		Fix $n\in \N$. For $\pv\in\nmod$ we have $d_\altmathcal{R}^1(\pv,\mathdszero)\leq n\cdot {n\choose 2} = n^2(n+1)/2$.
		\begin{proof}
			There are precisely ${n\choose k}$ intervals in $\{1,\dots, n\}$ corresponding to non-simple summands. Note that any sequence of reflections which annihilates an interval summand of $\pv$ annihilates all copies of this summand. Since any interval summand of a zigzag module can be annihilated with no more than $n$ reflections, we have $d_\altmathcal{R}^1(\pv,\mathdszero)\leq n\cdot {n\choose 2}$.
		\end{proof}
	\end{lemma}

	\begin{lemma}\label{lem:d_b=0_implies_d_R=0}
		Fix $n\in \N$ and a type $\tau\in\altmathcal{T}_n$. For $\pv,\pw\in\taumod{\tau}$, if $d_b^1(\dgm(\pv),\dgm(\pw)) = 0$ then $d_\altmathcal{R}(\pv,\pw) = 0$. 
		\begin{proof}
			It is not hard to see that if $d_b^1(\dgm(\pv),\dgm(\pw)) = 0$ then 
			\[ \lms(b,d)\in \dgm(\pv) \ | \ b\neq d\rms = \lms(b,d)\in \dgm(\pw) \ | \ b\neq d\rms,\]
			i.e. $\dgm(\pv)$ and $\dgm(\pw)$ have the same multisubsets of off-diagonal points. It follows that $\altmathcal{S}(\pv) = \altmathcal{S}(\pw)$ so that $\epsilon(\pv) = \altmathcal{S}(\pv) = \altmathcal{S}(\pw) \precsim \pw$ and $\epsilon(\pw) = \altmathcal{S}(\pw) = \altmathcal{S}(\pv) \precsim \pv$. Hence $d_\altmathcal{R}^1(\pv,\pw) = 0$.
		\end{proof}
	\end{lemma}
	\begin{proposition}[Bi-Lipschitz Equivalence of $d_{\altmathcal{R}}^1$ and $d_b^1$ for fixed $n$] \label{prop:db1_bilip}Fix $n\in \N$ and a type $\tau\in\altmathcal{T}_n$. For all zigzag modules $\pv,\pw\in\taumod{\tau}$ we have
		\[ d_b^1(\dgm(\pv),\dgm(\pw))\leq d_{\altmathcal{R}}^1(\pv,\pw)\leq n^2(n+1)\cdot d_b^1(\dgm(\pv),\dgm(\pw)).\]
		\begin{proof} The first of these inequalities is the claim of Theorem \ref{thm:main}. Now let $\eta = d_{\altmathcal{R}}^1(\pv,\pw)$. If $\eta = 0$ then the second inequality holds trivially, so suppose that $\eta>0$. By (the contrapositive of) Lemma \ref{lem:d_b=0_implies_d_R=0}, we have $d_b^1(\dgm(\pv),\dgm(\pw))\geq 1$. By Lemma \ref{lem:bound_to_zero} and the triangle inequality, 
			\[ d_\altmathcal{R}^1(\pv,\pw) \leq d_\altmathcal{R}^1(\pv,\mathdszero) + d_\altmathcal{R}^1(\pw,\mathdszero) \leq n^2(n+1).\]
			Thus $d_\altmathcal{R}^1(\pv,\pw)\leq n^2(n+1)\leq n^2(n+1)\cdot d_b^1(\dgm(\pv),\dgm(\pw))$.
		\end{proof}
	\end{proposition}
	\begin{corollary}\label{cor:bottleneck_bilip} Fix $n\in \N$ and a type $\tau\in \altmathcal{T}_n$. For any $\pv,\pw\in \taumod{\tau}$ we have
		\[d_{b}^\infty(\dgm(\pv),\dgm(\pw)) \leq d_{\altmathcal{R}}^1(\pv,\pw)\leq 2n^2(n+1)\cdot d_b^\infty(\dgm(\pv),\dgm(\pw)).\]
		\begin{proof}
			This follows immediately from Proposition \ref{prop:db1_bilip} and the equivalence $d_b^\infty\leq d_b^1\leq 2d_b^1$ (see Remark \ref{remark:db1_bottleneck_equiv}).
		\end{proof}
	\end{corollary}
	\begin{remark}
		Bi-Lipschitz equivalence does not hold in general on $\nmod$. To see this, consider any two zigzag modules $\pv,\pw\in \nmod$ for which $\dgm(\pv) = \dgm(\pw)$ but $\pv\not\sim_\altmathcal{S}\pw$. Then $d_b^1(\dgm(\pv),\dgm(\pw)) = 0$ but $d_\altmathcal{R}^1(\pv,\pw)>0$ by the second statement of Theorem \ref{thm:ref_is_metric}. For example, let $n = 3$, let $\tau = (\rightarrow,\rightarrow)$, and let $\tau' = (\rightarrow,\leftarrow)$. Define
		\[\pv = \mathbb{I}_\tau([1,2]) \oplus \mathbb{I}_{\tau}([2,3]) \quad \textup{ and } \quad \pw = \mathbb{I}_{\tau'}([1,2]) \oplus \mathbb{I}_{\tau'}([2,3]).\]
		By definition $\dgm(\pv) = \dgm(\pw)$. On the other hand, we have $\altmathcal{S}(\pv) = \pv$ and $\altmathcal{S}(\pw) = \pw$ so that $\pv \sim_{\altmathcal{S}} \pw$ if and only if $\,\pv\precsim \pw$ and $\pw \precsim \pv$ which is true if and only if $\,\pv \sim \pw$. But $\pv\not\sim \pw$, as none of the structure maps of $\pv$ nor $\pw$ are isomorphisms. Thus $\pv\not\sim_{\altmathcal{S}}\pw$.
	\end{remark}
	
	\subsection{Comparison with Botnan and Lesnick's Distance}
	Botnan and Lesnick introduced a distance between zigzag modules in \cite{algstabzz}. Their distance is defined by considering extensions of zigzag modules to so-called \textit{block decomposable}, two-dimensional persistence modules. Their distance, denoted $d_I$, is defined to be the two-dimensional interleaving distance between these extensions. In this section we consider examples which show that in general there is not a bi-Lipschitz equivalence between $d_I$ and $d_\altmathcal{R}^1$. We refer the reader to \cite[Definition 4.4]{algstabzz} for details on $d_I$.
	
	\begin{example} Let $n = 3$ and let $\tau = (\leftarrow,\rightarrow)\in\altmathcal{T}_n$. Consider the interval modules $\pv:= \mathbb{I}_\tau([1,3])$. We have $d_\altmathcal{R}^1(\pv,\mathdszero) = 2$ by Remark \ref{finiteness_remark}. On the other hand, we can view $\pv$ as the zigzag module
		\[\widetilde{\pv} = \begin{tikzcd}[column sep=2em]
		\cdots &\arrow[l,swap,"0"]0 \arrow[r,"0"]& \mathbb{F}& \mathbb{F} \arrow[r,"\textup{id}"]\arrow[l,swap,"\textup{id}"]& \mathbb{F}  & 0 \arrow[l,swap,"0"]\arrow[r,"0"]& \cdots,
		\end{tikzcd}\]
		a zigzag module indexed by the integers obtained from $\pv$ by extending with $0$. The block extension functor of Botnan and Lesnick sends $\widetilde{\pv}$ to an unbounded block in $\R^2$ and hence $d_I(\widetilde{\pv},\mathdszero) = \infty$ (see Figure 4 in \cite{algstabzz}). 
	\end{example}
	
	The next example gives a sequence $(\pv_n)_{n\in\N}$ of zigzag modules whose reflection distance to the zero module increases linearly but whose $d_I$ distance to the zero module is constant:
	
	\begin{example}
		For each $n\in \N$ let $\pv_n$ be the zigzag module
		\[\pv_n = \begin{tikzcd}[column sep=2em]
		\mathbb{F} &\arrow[l,swap,"\textup{id}"]\mathbb{F} \arrow[r,"0"]& 0& 0 \arrow[r,"0"]\arrow[l,swap,"0"]& \cdots\arrow[r,"0"]& \mathbb{F} &\arrow[l,swap,"\textup{id}"]\mathbb{F} \arrow[r,"0"]& 0& 0 \arrow[l,swap,"0"]
		\end{tikzcd}\]
		of length $4n$ ($\,\pv_n$ is the ``concatenation" of $n$ copies of $\,\mathbb{F} \stackrel{\textup{id}}{\longleftarrow} \mathbb{F}  \stackrel{0}{\longrightarrow} 0 \stackrel{0}{\longleftarrow}0$).
		It is not hard to see that $d_\altmathcal{R}^1(\pv_n,\mathdszero) = n$. On the other hand, extending each $\pv_n$ by 0 to obtain zigzag modules $\widetilde{V}_n$ indexed over the integers, we have $d_I(\widetilde{V}_n,\mathdszero) = 1/2$ for all $n$. To see this, note that the extension of $\widetilde{V}_n$ to a two-dimensional persistence module is a direct sum of block modules $I^{[1,2)_{\textup{BL}}},I^{[3,4)_{\textup{BL}}},\dots, I^{[2n-1,2n)_\textup{BL}}$ (see Section 4.1 of \cite{algstabzz} for notation). These block modules are infinite horizontal strips of height $1$ with vertical distance $1$ between adjacent blocks (see again Figure 4 from \cite{algstabzz}; $I^{[k,k+1)_\textup{BL}}$ is a vertical shift by $k$ of the block $I^{[1,2)_\textup{BL}}$). The zero morphism then serves as a $1/2$-interleaving between this direct sum of block modules and the zero module. Moreover, after a moment of reflection one sees that $1/2$ is the smallest possible interleaving constant.
	\end{example}

	\section{Discussion}\label{sec:discussion}

	Our definition of the reflection distance made  use of the notion of reflection functors introduced by Gelfand et al \cite{bgp} which are transformations on zigzag modules that affect only a portion a given module. In our constructions, the effect of these transformations was restricted to a subdiagram of length at most $3$ (cf. Section \ref{sec:reflections}, equation (\ref{subdiagram})). This design choice could also potentially be altered, and its exploration may lead to other interesting distances. In particular, it is conceivable that similar ideas can be exported to the setting of persistence modules over graphs other than those inducing zigzag modules.
	
	Indeed, since Gelfand et al consider reflection functors on arbitrary graphs, the possibility of extending our reflection distance to such general setting appears interesting.  One initial question arising from this is whether by applying a suitable sequence of reflections to a given persistence module defined on a graph, one can transform it into a summand of another such persistence module. Whereas we were able to show that this is always possible for zigzag modules,  it is not clear that this can be done for persistence modules defined on arbitrary graphs. 
	
	Our construction of a distance between zigzag modules took a  route different from the one followed by Botnan and Lesnick in \cite{algstabzz}. We showed that these two distances are in general not bi-Lipschitz equivalent. It seems of interest to identify large families of zigzag persistence modules for which a bi-Lipschitz equivalence might be possible.
	
	Another topic where research would be welcome is the elucidation of the computational complexity associated to estimating  distances between zigzag modules of a given length.  In this direction, recent results by Botnan, Bjerkevik, and Kerber \cite{kerber} about the computational complexity of the two-dimensional interleaving distance might be relevant.
	

	\subsection*{Acknowledgements}
	We thank Peter Bubenik, Mike Catanzaro, Woojin Kim, and Jos\'e Perea for their useful feedback. We also thank Lauren Wickman for her careful reading of the late drafts. A.E. was supported by the NSF RTG \# 1547357. F.M. was supported by NSF RI \# 1422400 and  NSF AF \#1526513.
	\bibliographystyle{alpha}
	\bibliography{bibfile}{}
	
	\appendix
	\appendixpage
	\addappheadtotoc
	\renewcommand{\thesection}{A\arabic{section}}

	\section{Limits and Colimits}\label{app:cats}
	We assume the reader is familiar with (small) categories, functors, and natural transformations (see \cite{categories}, \cite{riehl2017category} for details). We denote by $\mathbf{Vect}$ the category of vector spaces over some fixed field $\mathbb{F}$.
	\begin{definition}
		Fix a small category $\altmathcal{J}$.
		\begin{enumerate}
			\item 
			A \textit{diagram of vector spaces of shape} $\altmathcal{J}$ is a functor $D:\altmathcal{J}\to \textbf{Vect}$. 
			\item A \textit{cone} over a diagram $D:\altmathcal{J}\to \textbf{Vect}$ is a vector space $N$ together with a collection of linear transformations $\lambda :=\{\lambda_J: N\to D(J)\ | \ J\in\textup{Ob}(\altmathcal{J})\}$, indexed by the objects of $\altmathcal{J}$, such that for any morphism $f:A\to B$ in $\textup{Hom}(\altmathcal{J})$ we have $\lambda_B = D(f)\circ \lambda_A$. We denote such a cone by $(N,\lambda)$.
			\item A \textit{cocone} over a diagram $D:\altmathcal{J}\to \textbf{Vect}$ is a vector space $M$ together with a collection of linear transformations $\gamma :=\{\gamma_J: D(J)\to M \ | \ J\in\textup{Ob}(\altmathcal{J})\}$, indexed by the objects of $\altmathcal{J}$, such that for any morphism $f:A\to B$ in $\textup{Hom}(\altmathcal{J})$ we have $\gamma_A = \gamma_B\circ D(f)$. We denote such a cocone by $(M,\lambda)$.
		\end{enumerate}
	\end{definition}
	Limits and colimits are then \textit{universal cones} or \textit{cocones}, respectively:
	\begin{definition} Let $D:\altmathcal{J}\to\mathbf{Vect}$ be a diagram of vector spaces.
		\begin{enumerate}
			\item The \textit{limit} of the diagram $D$ is a cone $(L,\phi)$ over $D$ such that for any other cone $(N,\lambda)$ over $D$, there exists a unique linear transformation $\psi:N\to L$ with $\lambda_A = \phi_A \circ \psi$ for all $A\in \textup{Ob}(\altmathcal{J})$. We denote the limit $L$ by $\,\lim D$.
			\item The \textit{colimit} of the diagram $D$ is a cocone $(C,\phi)$ over $D$ such that for any other cocone $(M,\lambda)$ over $D$, there exists a unique linear transformation $\psi:M\to C$ with $\lambda_A = \phi_A \circ \psi$ for all $A\in \textup{Ob}(\altmathcal{J})$. We denote the colimit $C$ by $\colim(D)$.
		\end{enumerate}
	\end{definition}			
	
	\begin{theorem}\label{mono_diagrams_mono_lim} Let $\altmathcal{J}$ be a small category and let $D^1,D^2$ be diagrams of vector spaces. If there exists a natural transformation $\eta:D^1\rightarrow D^2$ all of whose components are monomorphisms, then $(L^1,\eta_J\circ \lambda^1_J)$ is a cone for $D^2$ and the unique morphism $\psi:L^1\to L^2$ satisfying $\eta_J\circ \lambda^1_J = \lambda^2_J\circ\psi$ for all $J\in\altmathcal{J}$ is a monomorphism.
		\begin{proof}
			The proof can be found in more generality in \cite{borceux1994handbook} pg. 89, Corollary 2.15.3.
		\end{proof}
	\end{theorem}
	
	Let $\mathbf{2}$ denote the discrete category with 2 objects. That is, $\textup{Ob}(\mathbf{2}) = \{1,2\}$ and $\textup{Hom}(1,2) = \emptyset$.
	\begin{definition}
		Let $X,Y$ be objects in the category $\,\altmathcal{C}$ and let $D:\mathbf{2}\to\altmathcal{C}$ be the diagram given by $D(1) = X$ and $D(2) = Y$.
		\begin{enumerate}
			\item The \textit{product} of $B$ and $C$, denoted $B\times C$, is the limit of $D$, if it exists,
			\item The \textit{coproduct} of $B$ and $C$, denoted $B\amalg C$, is the colimit of $D$, if it exists.
		\end{enumerate}
	\end{definition}
	In $\mathbf{Vect}$, products and coproducts always exist and coincide.
	
	\begin{theorem}\label{lim_prod_is_prod_lim} Let $\altmathcal{J}$ be a small category and let $D^1, D^2$ be diagrams of vector spaces. Then $\lim(D^1\times D^2)\cong \lim (D^1)\times \lim (D^2)$ Dually, $\colim(D^1\amalg D^2)\cong \colim(D^1)\amalg \colim(D^2)$. Here $D^1\times D^2$ and $D^1\amalg D^2$ denote the product and coproduct, respectively, of $D^1$ and $D^2$ in the functor category $\altmathcal{C}^\altmathcal{J}$.
		\begin{proof} Consider the product category $\mathbf{2}\times J$ and let $F:\mathbf{2}\times \altmathcal{J}\to\altmathcal{C}$ be the functor given by $F(i,j) = D^i(j)$ for $i = 1,2$ and $j\in\altmathcal{J}$. By commutativity of limits,
			\[\lim_\mathbf{2}\lim_\altmathcal{J} F(i,j) = \lim_\altmathcal{J}\lim_\mathbf{2}F(i,j)\]
			(see \cite{riehl2017category} pg. 111, Theorem 3.8.1). Now 
			\begin{align*}
			\lim(D^1\times D^2) &= \lim_{\altmathcal{J}}(D^1\times D^2) = \lim_{\altmathcal{J}}\lim_\mathbf{2} F(i,j)\\
			& = \lim_{\mathbf{2}}\lim_\altmathcal{J} F(i,j) = \lim_{\mathbf{2}}\lim D^i\\
			& = \lim(D^1)\times \lim(D^2).
			\end{align*}
			The second statement follows by a duality argument.
		\end{proof}
	\end{theorem}
	
	
	\section{Matchings}\label{app:matchings}
	The Matching Lemma is a useful tool for combining matchings in opposite directions into a single matching. For us, the lemma was crucial for proving our stability result, while Bjerkevik makes use of the lemma in proving his generalization of the algebraic stability theorem \cite{bjerkevik}. While it may be one of the earliest results in infinite matching theory \cite{AHARONI19915}, we believe its applications to stability-type questions in persistence theory make it worth expounding upon here.
	\begin{lemma}[Matching Lemma for Finite Matchings]
		Let $S$ and $T$ be sets and let $f:S\nrightarrow T$ and $g:T\nrightarrow S$ be finite matchings. Then there exists a matching $M:S\nrightarrow T$ such that 
		\begin{enumerate}[label=(\arabic*)]
			\item $\coim(f)\subseteq\coim(M)$,
			\item $ \coim(g)\subseteq\im(M)$,
			\item if $M(s) = t$ then either $f(s) = t$ or $g(t) = s$.
		\end{enumerate} 	
		\begin{proof} The proof is inspired by a proof of the Cantor–Schr\"{o}der–Bernstein theorem \footnote{The infinite version of the Matching Lemma holds as well \cite{ore1962theory} and actually generalizes the  Cantor–Schr\"{o}der–Bernstein (CSB) theorem slightly; the CBS theorem says that if $f:S\to T$ and $g:T\to S$ are injections between sets $S$ and $T$ then there exists a bijection $M:S\to T$. Since injective maps can be viewed as matchings with coimages being equal to the domain of the map, the infinite Matching Lemma implies the existence of a matching $M:S\nrightarrow T$ with $\coim(M) = \coim(f) = S$ and $\im(M) = \coim(g) = T$, i.e. a bijection between $S$ and $T$.}
			given in \cite{CSB}. Let $s\in \coim(f)$, $t\in \coim(g)$, and consider sequences of the form
			\[ \cdots \stackrel{g}{\longrightarrow}f^{-1}(g^{-1}(s))\stackrel{f}{\longrightarrow} g^{-1}(s)\stackrel{g}{\longrightarrow}s\stackrel{f}{\longrightarrow} f(s) \stackrel{g}{\longrightarrow} g(f(s)) \stackrel{f}{\longrightarrow} \cdots \]
			and
			\[ \cdots \stackrel{f}{\longrightarrow}g^{-1}(f^{-1}(t))\stackrel{g}{\longrightarrow} f^{-1}(t)\stackrel{f}{\longrightarrow}t\stackrel{g}{\longrightarrow} g(t) \stackrel{f}{\longrightarrow} f(g(t)) \stackrel{g}{\longrightarrow} \cdots, \]
			where we allow these sequences to terminate to the right or left when undefined. We refer to these sequences as the \textit{orbits} of $s$ or $t$. Since $f$ and $g$ are finite matchings, such sequences either terminate on both the left and right, or are infinite but periodic. By injectivity of $f$ and $g$, every element of $\coim(f)\sqcup\coim(g)$ appears in exactly one orbit. Moreover, every orbit falls into one of the following five classes:
			\begin{enumerate}
				\item $s\rightarrow t\rightarrow \cdots \rightarrow s\rightarrow y$,
				\item $s\rightarrow t\rightarrow \cdots \rightarrow s\rightarrow t\rightarrow x$,
				\item $t\rightarrow s\rightarrow \cdots \rightarrow t\rightarrow x$,
				\item $t\rightarrow s\rightarrow \cdots \rightarrow t\rightarrow s\rightarrow y$,
				\item 
				$\begin{tikzcd}[row sep=1em,column sep=1em]
				s \arrow[r] & t \arrow[r]&s \arrow[r] & t \arrow[r]& s \arrow[d]\\
				t \arrow[u] &&&& t \arrow[d]\\
				s \arrow[u] & t \arrow[l]&\cdots \arrow[l] & t \arrow[l]& s \arrow[l] 
				\end{tikzcd}$
			\end{enumerate}
			where the $s$'s and $t$'s represent elements of $\coim(f)$ and $\coim(g)$, respectively, arrows represent either $f$ or $g$, and $y$'s and $x$'s represent elements of $\im(f)\setminus \coim(g)$ and $\im(g)\setminus \coim(f)$, respectively. We define a matching $M:S\nrightarrow T$ as follows: for each $i = 1,\hdots,5$ let
			\[ S_i^{\textup{co}}:= \{s\in \coim(f) \ | \ \textup{$s$ appears in an orbit of type $i$}\}\]
			and 
			\[ T_i^{\textup{co}}:= \{t\in \coim(g) \ | \ \textup{$t$ appears in an orbit of type $i$}\}.\]
			Then $\coim(f) = \bigcupdot S_i^{\textup{co}}$ and $\coim(g) = \bigcupdot T_i^{\textup{co}}$. We partition $S_1^{\textup{co}}$ and $T_3^{\textup{co}}$ further by defining
			\[ S_1^{\textup{co$\backslash$p}} : = \{s\in S_1^{\textup{co}} \ | \ f(s)\in\coim(g)\}, \quad S_1^{\textup{p}} : = \{s\in S_1^{\textup{co}} \ | \ f(s)\not\in\coim(g)\},\]
			\[ T_3^{\textup{co$\backslash$p}} : = \{t\in T_3^{\textup{co}} \ | \ g(t)\in\coim(f)\}, \quad \textup{ and } \quad T_3^{\textup{p}} : = \{t\in T_3^{\textup{co}}\ | \ g(t)\not\in\coim(f)\},\]
			so that $S_1 = S_1^{\textup{co$\backslash$p}} \cupdot S_1^{\textup{p}}$ and $T_3 = T_3^{\textup{co$\backslash$p}} \cupdot T_3^{\textup{p}}$. The image, coimage, and mapping of $M$ is specified by the diagram			
			\tikzset{%
				symbol/.style={%
					draw=none,
					every to/.append style={%
						edge node={node [sloped, allow upside down, auto=false]{$#1$}}}
				}
			}
			\[	\newcommand{\defequals}{:=}
			\begin{tikzcd}[row sep=2em,column sep=1.9em]
			\coim(M)\arrow[d,"M"] \arrow[symbol=\defequals]{r}&S_1^{\textup{co$\backslash$p}}\arrow[symbol=\bigcupdot]{r} \arrow[d,"f"]&S_2^{\textup{co}}\arrow[symbol=\bigcupdot]{r}\arrow[d,"f"]&S_3^{\textup{co}}\arrow[symbol=\bigcupdot]{r}\arrow[d,"g^{-1}"]&S_4^{\textup{co}}\arrow[symbol=\bigcupdot]{r}\arrow[d,"g^{-1}"]&S_5^{\textup{co}}\arrow[symbol=\bigcupdot]{r}\arrow[d,"f"]&S_1^{\textup{p}}\arrow[symbol=\bigcupdot]{r}\arrow[dr,"f",swap,pos=0]&g(T_3^{\textup{p}})\arrow[dl,"g^{-1}",crossing over,pos= 0]\\
			\im(M) \arrow[symbol=\defequals]{r}&T_1^{\textup{co}}\arrow[symbol=\bigcupdot]{r} &T_2^{\textup{co}}\arrow[symbol=\bigcupdot]{r}&T_3^{\textup{co$\backslash$p}}\arrow[symbol=\bigcupdot]{r}&T_4^{\textup{co}}\arrow[symbol=\bigcupdot]{r}&T_5^{\textup{co}}\arrow[symbol=\bigcupdot]{r}&T_3^{\textup{p}}\arrow[symbol=\bigcupdot]{r}&f(S_1^{\textup{p}})
			\end{tikzcd}
			\]
			\hfill
			
			Note that $f(S_1^{\textup{co$\backslash$p}}) = T_1^{\textup{co}}$, $f(S_2^{\textup{co}}) = T_2^{\textup{co}}$, $g(T_3^{\textup{co$\backslash$p}}) = S_3^{\textup{co}} $, $g(T_4^{\textup{co}}) = S_4^{\textup{co}}$, and $f(S_5^{\textup{co}}) = T_5^{\textup{co}}$ so that $M$ is surjective. Since $f$ and $g$ are injective, we see that $M$ is a bijection between each of the parts specified and thus defines a matching. Moreover, $\coim(M) = \coim(f)\cupdot g(T_3^{\textup{p}})\supseteq \coim(f)$ and $\im(M) = \coim(g)\cupdot f(S_1^{\textup{p}})\supseteq \coim(g)$. Property (3) evidently holds by the definition of $M$.
		\end{proof}
	\end{lemma}
\end{document}